\documentclass[12pt,reqno,twoside]{amsart}
\usepackage{tikz}
\usepackage{amssymb}
\usepackage{amsmath}
\usepackage{amsthm}
\usepackage{amsfonts}
\usepackage{amsthm,amscd}
\usepackage{amsbsy}
\usepackage{bm}
\usepackage{url}  
\usepackage{setspace}
\usepackage{float}
\usepackage{psfrag}
\usepackage{tikz}
\usetikzlibrary{calc,intersections}
\usepackage{float}
\usepackage{epstopdf}
\usepackage{caption}
\usepackage{graphics}
\usepackage{pgfplotstable}
\usepackage{multicol}
\usepackage{pgffor}

\allowdisplaybreaks

\AtBeginDocument{
  \def\MR#1{}
}

\usetikzlibrary{positioning}

\usepackage[colorlinks, linkcolor=blue]{hyperref}

\restylefloat{figure}

\newtheorem{Theorem}{Theorem}[section]
\newtheorem{Lemma}[Theorem]{Lemma}

\newtheorem{Proposition}[Theorem]{Proposition}
\newtheorem{Remark}[Theorem]{Remark}

\newtheorem{Definition}[Theorem]{Definition}

\numberwithin{equation}{section}
\usepackage{setspace}
\usepackage{anysize}
\usepackage[margin=1in]{geometry}
\usepackage[noadjust]{cite}
\def\be{\begin{equation}}
	\def\ee{\end{equation}}
\def\ben{\begin{eqnarray}}
	\def\een{\end{eqnarray}}

\newcommand{\ncom}{\newcommand}

\hypersetup{colorlinks=true,%
	citecolor=red,%
	filecolor=blue,%
	linkcolor=blue,%
}

\ncom{\n}{\normalfont}

\ncom{\Lc}{\mathcal}
\ncom{\wt}{\widetilde}

\newcommand{\vertiii}[1]{{\left\vert\kern-0.25ex\left\vert\kern-0.25ex\left\vert #1 
		\right\vert\kern-0.25ex\right\vert\kern-0.25ex\right\vert}}

\long\def\/*#1*/{}
\title[Error estimates for Burgers' equation using deep learning method]{Error estimates for viscous Burgers' equation using deep learning method}
\date{\today}

\author{{Wasim Akram$^\dag$, Sagar Gautam$^\dag$, Deepanshu Verma$^\ddag$}\and{Manil T. Mohan$^\dag$}}

\thanks{$\dag$ Department of Mathematics, Indian Institute of Technology Roorkee, Uttarakhand, 247667, India, Email- {\normalfont{ wakram2k11@gmail.com; sagargautamkm@gmail.com; maniltmohan@ma.iitr.ac.in}}\\
$\ddag$ Department of Mathematical and Statistical Sciences, Clemson University, Clemson, 29631, USA, Email- {\normalfont {dverma@clemson.edu}}\\
Wasim Akram is supported by NBHM (National Board of Higher Mathematics, Department of Atomic Energy) postdoctoral fellowship, No. 0204/16(1)(2)/2024/R\&D-II/10823. }

\makeatletter
\renewcommand{\tocsection}[3]{%
	\indentlabel{\@ifnotempty{#2}{\bfseries\ignorespaces#1 #2\quad}}\bfseries#3}
\renewcommand{\tocsubsection}[3]{%
	\indentlabel{\@ifnotempty{#2}{\ignorespaces#1 #2\quad}}#3}

\newcommand\@dotsep{4.5}
\def\@tocline#1#2#3#4#5#6#7{\relax
	\ifnum #1>\c@tocdepth 
	\else
	\par \addpenalty\@secpenalty\addvspace{#2}%
	\begingroup \hyphenpenalty\@M
	\@ifempty{#4}{%
		\@tempdima\csname r@tocindent\number#1\endcsname\relax
	}{%
		\@tempdima#4\relax
	}%
	\parindent\z@ \leftskip#3\relax \advance\leftskip\@tempdima\relax
	\rightskip\@pnumwidth plus1em \parfillskip-\@pnumwidth
	#5\leavevmode\hskip-\@tempdima{#6}\nobreak
	\leaders\hbox{$\m@th\mkern \@dotsep mu\hbox{.}\mkern \@dotsep mu$}\hfill
	\nobreak
	\hbox to\@pnumwidth{\@tocpagenum{\ifnum#1=1\bfseries\fi#7}}\par
	\nobreak
	\endgroup
	\fi}
\AtBeginDocument{%
	\expandafter\renewcommand\csname r@tocindent0\endcsname{0pt}
}
\def\l@subsection{\@tocline{2}{0pt}{2.5pc}{5pc}{}}
\makeatother

\usepackage[tickmarkheight=.5em,textwidth=\marginparwidth,textsize=small]{todonotes}
\begin{document}
	\begin{abstract}
The article focuses on error estimates as well as stability analysis of deep learning methods for stationary and non-stationary viscous Burgers equation in two and three dimensions. The local well-posedness of homogeneous boundary value problem for non-stationary viscous Burgers equation is established  by using semigroup techniques and fixed point arguments. By considering a suitable approximate problem and deriving appropriate energy estimates, we prove the existence of a unique strong solution. Additionally, we extend our analysis to the global well-posedness of the non-homogeneous problem. For both the stationary and non-stationary cases, we derive explicit error estimates in suitable Lebesgue and Sobolev norms by optimizing a loss function in a Deep Neural Network approximation of the solution with fixed complexity. Finally, numerical results on prototype systems are presented to illustrate the derived error estimates.
\end{abstract}
	\maketitle
	\pagenumbering{arabic}

\noindent \textbf{MSC Classification (2020).} 65N15  $\cdot$ 68T07  $\cdot$ 35A01  $\cdot$ 35A02. 


\section{Introduction}
The Burgers equation is a fundamental partial differential equation (PDE) that arises in various fields, including fluid dynamics, gas dynamics, traffic flow, and turbulence modeling. Serving as a simplified model for the more complex Navier-Stokes equations, it provides valuable insights into nonlinear advection-diffusion processes. The study of the Burgers equation in two and three dimensions is essential for understanding diffusion-driven wave dynamics, shock formation, and wavefront propagation which are key phenomena in real-world physical systems.

\subsection{Motivation}
There are various numerical methods for solving PDEs, including the finite difference method, finite element method, and finite volume method. However, these traditional approaches often come with limitations, such as high computational cost and long execution times. In recent years, machine learning methods, particularly deep neural networks (DNNs), have shown promise in addressing these challenges. Among these, Physics-Informed Neural Networks (PINNs) \cite{KDKS19} have gained significant attention for solving PDEs efficiently. One of the key advantages of PINNs is their ability to approximate a PDE's solution without requiring a discretized grid. This is achieved by training a neural network to satisfy both the PDE and the associated boundary/initial conditions.

Deep learning \cite{HighhamDNN} is a branch of machine learning that focuses on developing models inspired by the structure and functioning of the human brain. It builds on artificial neural networks to create systems capable of interpreting and learning from data. Multi-layer perceptrons, which consist of multiple hidden layers, are commonly used in deep learning. Figure~\ref{fig:algo} provides an illustration of a four-layer neural network used for the PINN approximation of Burgers equation. In this setup, the network is supervised, meaning it requires a “teacher” to guide it by providing the desired outputs. The system constructs high-level representations by combining simpler, lower-level features, forming more abstract concepts. This hierarchical architecture consists of an input layer, several hidden layers, and an output layer, with connections existing only between neighboring layers. Each layer can act as a logistic regression model, enabling the network to develop complex concepts from simpler building blocks, making deep learning a powerful and flexible tool for data interpretation.

A typical dense neural network computes layer-by-layer transformations for a given input $x$:
$$
y_2(x) = \sigma(W_1 x + b_1),\quad  
y_3(y_2) = \sigma(W_2 y_2 + b_2), \dots
$$
leading to the final output:
$$
u_\theta(x) = y_{n+1}(y_n(\dots y_2(x))).
$$
Here, $ \theta = \{W_1, W_2, \dots, W_n, b_1, \dots, b_n\} $ represents all the weights and biases, and $\sigma$ is a nonlinear activation function and $n$ denote the number of layers. The objective of neural network training is to ``learn" or optimize the parameters $\theta$ by minimizing a chosen loss function. 

Throughout the article, $\mathcal{F}_N$ denote the class of functions representable by a deep neural network of complexity $N$, which constitutes a finite-dimensional subspace of functions over the bounded domain $\Omega$.

The primary goal of this paper is to provide a mathematically rigorous analysis of the PINN-approximated solution, including error estimates and stability analysis, for the viscous Burgers equations in two and three dimensions.   While hyperparameter tuning and exploring various neural network architectures are important aspects of solving PDEs with neural networks, they are not the focus here. In particular, we use smooth activation functions such as the sigmoid ($\sigma(x) = \frac{e^x}{e^x + 1}$), or the hyperbolic tangent $\tanh$ ($\tanh(x)= \frac{e^x -e^{-x}}{e^x+e^{-x}})$. Our approach is similar in spirit to the probabilistic error analysis for machine learning algorithms applied to the Black-Scholes equations \cite{BiswasDNN-NSE}. We examine two settings: the stationary case, which serves to establish fundamental ideas and illustrate our methodology, and the non-stationary case, which is the primary focus of this work. A detailed summary of our key contributions is provided in Section~\ref{sec:contr}.

\begin{figure}
\begin{tikzpicture}
	\draw[fill=purple!10] (-0.5,-2.7) rectangle (9.8,2.7);
	\tikzstyle{neuron} = [circle, draw=black, fill=orange!50, minimum size=20pt, inner sep=0pt]
	\tikzstyle{input neuron} = [neuron, fill=blue!20];
	\tikzstyle{output neuron} = [neuron, fill=green!30];
	\tikzstyle{hidden neuron} = [neuron, fill=orange!50];
	\tikzstyle{annot} = [text width=4em, text centered];
	\tikzstyle{derivative} = [neuron, fill=green!50];

	\node[input neuron] (I-1) at (0,1) {$x_1$};
	\node[input neuron] (I-2) at (0,0) {$x_2$};
	\node[input neuron] (I-3) at (0,-1) {$t$};
	
	\foreach \name / \y in {1,2,...,5}
	\path[yshift=0.5cm]
	node[hidden neuron] (H1-\name) at (2.5,2.5-\y) {$y_1^{(\name)}$};
	
	\foreach \name / \y in {1,2,...,5}
	\path[yshift=0.5cm]
	node[hidden neuron] (H2-\name) at (5,2.5-\y) {$y_2^{(\name)}$};
	
	\foreach \name / \y in {1,2,...,5}
	\path[yshift=0.5cm]
	node[hidden neuron] (H3-\name) at (7.5,2.5-\y) {$y_3^{(\name)}$};
	
	\node[output neuron] (O) at (9.5,0) {$u_\theta$};
	
	\node[derivative] (D-1) at (11,2) {$ u_\theta$};
	\node[derivative] (D-2) at (11,0.7) {$u_{\theta_t}$};
	\node[derivative] (D-3) at (11,-0.6) {$u_{\theta_{x_i}}$};
	\node[derivative] (D-4) at (11,-2) {$\Delta u_\theta$};
	
	\foreach \source in {1,2,3}
	\foreach \dest in {1,2,...,5}
	\path (I-\source) edge (H1-\dest);
	
	\foreach \source in {1,2,...,5}
	\foreach \dest in {1,2,...,5}
	\path (H1-\source) edge (H2-\dest)
	(H2-\source) edge (H3-\dest);
	
	\foreach \source in {1,2,...,5}
	\path (H3-\source) edge (O);
	
	\foreach \source in {1}
	\foreach \dest in {1,2,3,4}
	\path (O) edge (D-\dest);
	
	\node[rotate=90, draw, rectangle, text width=5cm, align=center, fill=yellow!30] (Eq) at (13.5, 0) {Calculate loss = \\ $\displaystyle \| u_{\theta_t} -\nu \Delta u_\theta +\sum_{i=1}^d u_\theta u_{\theta_{x_i}}-f\|^2$ \\ $+\|u_\theta|_\Gamma\|^2+\|u_\theta(0)-u_0\|^2$ };

	\foreach \source in {1,2,...,4}
	\path (D-\source) edge (Eq);
	
	\node[draw, rectangle, text width=3cm, align=center, fill=blue!10] (It) at (9,-3.5) {loss$<\epsilon$ or \\ $\ge$ Max iteration?};
	
	\draw[] (It) -- (5,-3.5);
	\draw[->] (5,-3.5) -- (5,-2.7);
	\node at (6.2,-3.7) {No};
	\draw[->] (13,-2.7) -- ++(0,0) |- (It);
	\node at (9.5,-5) {Yes};
	\draw[->] (9,-4) -- ++(0,0) |- (11.5,-5.5);
	\node at (12,-5.5) {Done};

	\node[annot, above of=H1-1, node distance=1.2cm] {Hidden Layer 1};
	\node[annot, above of=H2-1, node distance=1.2cm] {Hidden Layer 2};
	\node[annot, above of=H3-1, node distance=1.2cm] {Hidden Layer 3};
	
\end{tikzpicture}
\caption{An illustration of the PINN algorithm for solving a PDE, with the loss function specifically formulated for the Burgers equation.} \label{fig:algo}
\end{figure}
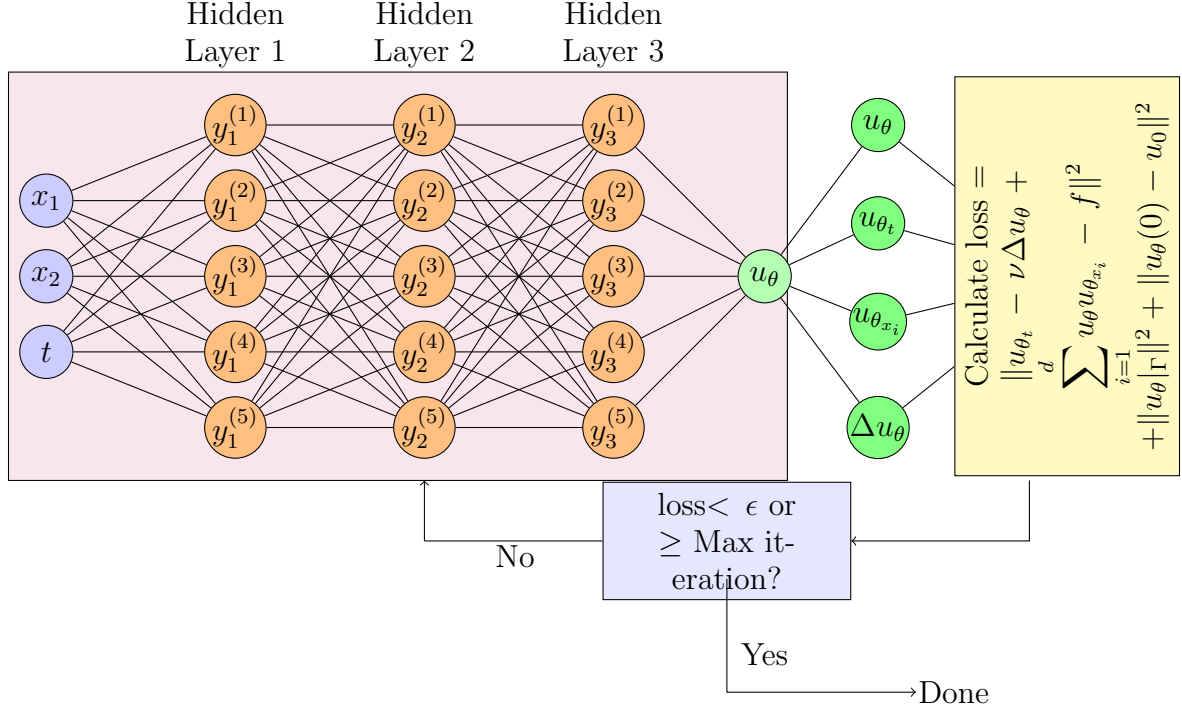

\subsection{Literature survey} 
Given the extensive applications of the Burgers equation, numerous numerical methods have been developed to solve it. The finite element method, in particular, has been widely utilized, as documented in studies such as \cite{AKR,Cald81, PanyNN07, AKPNN08, Dogan04, YChen19} and references therein. Error estimates for the one-dimensional Burgers equation were derived and verified through numerical experiments in \cite{PanyNN07}. The study in \cite{AKPNN08} provides an optimal error estimate accompanied by a comprehensive numerical analysis for the one-dimensional case. Additionally, \cite{YChen19} applied a weak Galerkin finite element method to achieve optimal-order error estimates for the one-dimensional Burgers equation, examining both semi-discrete and fully discrete systems, with numerical verification provided.

\medskip 
The finite element method is a standard approach for solving PDEs, while the success of deep neural networks in various approximation tasks has led to the development of PINNs for this purpose. PINNs and their variants have shown strong capabilities in approximating a wide range of PDEs. However, to date, PINNs and the finite element method have primarily been explored independently \cite{GKLSFEMPINN,Edu25JCAM}. There are vast literature available where PINN is used to solve partial differential equations, for example, see \cite{BiswasDNN-NSE,SMishMC,SMishACT,SMishIMA,KDKS19}, and references their in. The work \cite{SMishMC} established error bounds for ReLU neural network approximations of parametric hyperbolic scalar conservation laws, including generalization error bounds that depend on training error, sample size, and network size. The work \cite{BonIMA} presents a rigorous \emph{a priori} analysis of collocation methods, including PINNS, for elliptic PDEs, establishing optimal recovery rates under Besov space assumptions and proposing consistent loss functions that provide sharp error control and improved numerical performance.
 The authors in \cite{SMishIMA} provided error bounds for approximating the incompressible Navier-Stokes equations with PINN, proving that the residual error can be minimized with $\tanh$ networks, and that total error depends on training error, network size, and quadrature points.

\subsection{Contributions}\label{sec:contr}
This work provides a rigorous mathematical foundation for PINNs applied to stationary and non-stationary viscous Burgers equations in two and three dimensions. Our main contributions are:

\begin{itemize}
    \item \textbf{Sobolev space DNN approximation theory.} We extend existing neural network approximation results from $C^k$ functions to $H^m(\Omega) \cap H^1_0(\Omega)$, proving that deep neural networks can approximate Sobolev functions with arbitrarily small error in the Sobolev norm.
    
    \item \textbf{Well-posedness results.} We establish explicit solvability for both stationary and time-dependent viscous Burgers equations under homogeneous and non-homogeneous boundary conditions, covering weak and strong solutions.
    
    \item \textbf{Error estimates and stability.} We derive sharp bounds linking the residual error to the error of the solution, quantify the impact of boundary regularity and prove stability results showing that two PINN solutions remain close when their respective input data are close.
    
    \item \textbf{Constructive existence of PINN approximants.} For both stationary and non-stationary cases, we prove the existence of neural networks achieving arbitrarily small residual and solution errors in the Sobolev norm.
    
    \item \textbf{Numerical validation.} We present experiments for both stationary and time-dependent problems, confirming the theoretical error bounds and demonstrating the effectiveness of our deep learning approach.
\end{itemize}

\section{Functional setting and mathematical framework} \label{sec:funcSet}
We provide the function spaces needed to obtain the required results and some basic inequalities in this section. 
\subsection{Functional setting}  
Let $\Omega$ be a bounded domain in $\mathbb{R}^d$. Let $C_c^{\infty}(\Omega)$ be the space of all $C^\infty$ functions with compact support contained in $\Omega$. For $1\le p<\infty,$ let  $L^p(\Omega)$ denote the space of equivalence classes of Lebesgue measurable functions $f : \Omega \to\mathbb{R}$ such that $\int_{\Omega}|f(x)|^pdx<\infty$, where two measurable functions are equivalent if they are equal a.e. For $f\in L^p(\Omega)$, we define the $L^p-$norm of $f$ as $\|f\|_{L^p(\Omega)}:=\left(\int_{\Omega}|f(x)|^pdx\right)^{1/p}$. We define $L^{\infty}(\Omega)$, the equivalence class of Lebesgue measurable functions $f:\Omega\to\mathbb{R}$ for which $\|f\|_{L^{\infty}(\Omega)}:=\mathrm{ess \ sup}_{x\in\Omega}|f(x)|<\infty$. For $p=2$, $L^2(\Omega)$ is a Hilbert space and the inner product in  $L^2(\Omega)$  is denoted by $(\cdot,\cdot)$ and is given by $(f,g)=\int_{\Omega}f(x)g(x)dx$, and the corresponding norm is denoted by $\|\cdot\|$. 	Moreover, let $H_0^1(\Omega)$ denote the Sobolev space (also denoted as $W^{1,2}_0(\Omega)$), which is defined as the space of equivalence classes of Lebesgue measurable functions $f \in L^2(\Omega)$ such that its weak derivative $\nabla f \in L^2(\Omega)$ and $f$ is of trace zero.  The $H_0^1(\Omega)-$norm is defined by $\|{f}\|_{H_0^1(\Omega)} := \left(\int_{\Omega}|\nabla f(x)|^2dx\right)^{1/2}$ (by using the Poincar\'e inequality).
Then, we define  $H^{-1}(\Omega) := (H_0^1(\Omega))'$, that is, dual of the Sobolev space $H_0^1(\Omega)$, with norm \[\|{g}\|_{H^{-1}(\Omega)} := \sup\left\{{\langle g, f\rangle}:  f \in H_0^1(\Omega),\ {\|{f}\|_{H_0^1(\Omega)}}\leq 1\right\},\] where $\langle\cdot,\cdot\rangle$ represent the duality pairing between $H_0^1(\Omega)$ and $H^{-1}(\Omega)$. 
We denote the second order Hilbertian Sobolev spaces by $H^2(\Omega)$. The general Sobolev spaces will be denoted by $W^{m,p}(\Omega)$, where $m\in\mathbb{N}$ and $p\in[1,\infty)$. It is known that there exists  a linear continuous operator $\gamma_0 \in\mathcal{L}(H^1(\Omega),L^2(\Gamma))$ (the trace operator), such that $\gamma_0 u$ equals the restriction of $u$ to $\Gamma,$ for every function $u \in H^1(\Omega)$  which is twice continuously differentiable in $\overline{\Omega}$, see \cite[p. 9]{TemamNSE01}. The kernel of $\gamma_0$ is equivalent to the space $H_0^1(\Omega)$. The image space $\gamma_0(H^1(\Omega))$ is a dense subspace of $L^2(\Gamma)$, which is represented by $H^{1/2}(\Gamma)$, that is, 
\begin{align*}
	H^{1/2}(\Gamma)=\gamma_0(H^1(\Omega))=\{v\in L^2(\Gamma): \text{ there exists a }  u\in H^1(\Omega) \mbox{ such that } \gamma_0(u)=v\}.
\end{align*}
The space $H^{1/2}(\Gamma)$ can be equipped with the norm carried from $H^1(\Omega)$ by $\gamma_0$.
The dual space of $H^{1/2}(\Gamma)$ is denoted by $H^{-1/2}(\Gamma)$.  Moreover, there exists  a linear continuous operator $\ell_{\Omega}\in \mathcal{L}(H^{1/2}(\Gamma),H^1(\Omega))$, which is called a \emph{lifting operator}, such that $\gamma_0\circ \ell_{\Omega}=$ the identity operator in $H^{1/2}(\Gamma)$. Similarly, one can define $H^{3/2}(\Gamma)=\gamma_0(H^2(\Omega))$. For further information, see Lions $\& $ Magenes \cite{LioMag-V2} and McLean \cite{McLean-2000}. 

The norm in any space $X$ is denoted by $\|\cdot\|_X.$  The constant $C>0$ is generic and may depend on the given parameters and domain.

\subsection{Some useful inequalities}
\noindent Let us
recall some useful inequalities that are frequently used in the paper.

\noindent \textbf{Young's inequality.} Let $a,b$ be any non-negative real numbers. Then for any $\varepsilon>0,$ the following inequalities hold:
\begin{align} \label{eqPR-YoungIneq}
	ab \le \frac{\varepsilon a^2}{2} + \frac{b^2}{2\varepsilon} \text{ and } ab \le \frac{a^p}{p}+\frac{b^q}{q},
\end{align}
for any $p,q>1$ such that $\frac{1}{p}+\frac{1}{q}=1.$
\begin{Proposition}[Generalized H\"older's inequality] \label{ppsPR-GenHoldIneq}
Let $\Omega$ be a bounded domain in $\mathbb{R}^d$.	Let $f\in L^p(\Omega),$ $g\in L^q(\Omega),$ and $h\in L^r(\Omega),$ where $1\le p, q ,r\leq\infty$ are such that $\frac{1}{p}+\frac{1}{q}+\frac{1}{r}=1.$ Then $fgh\in L^1(\Omega)$ and 
	\begin{align*}
		\|fgh\|_{L^1(\Omega)} \le \|f\|_{L^p(\Omega)} \|g\|_{L^q(\Omega)} \|h\|_{L^r(\Omega)}.
	\end{align*} 
\end{Proposition}

\begin{Lemma}[Agmons' inequality {\n\cite{Agmon10}}] \label{lemVB:AgmonIE}
	Let $\Omega$ be a bounded domain in $\mathbb{R}^d,\, d\in \{1,2,3\},$  and let $f \in H^{s_2}(\Omega).$ Let $s_1,s_2$ be such that $s_1<\frac{d}{2}<s_2.$ If for $0<\theta<1,$ $\frac{d}{2}=\theta s_1+(1-\theta)s_2,$ then there exists a positive constant $C_a=C_a(\Omega)$ such that
	\begin{equation} \label{eqVB-AgmonIE}
		\|f\|_{L^\infty(\Omega)} \le C_a \|f\|_{H^{s_1}(\Omega)}^{\theta} \|f\|_{H^{s_2}(\Omega)}^{1-\theta}.
	\end{equation}
\end{Lemma}

In the following lemma, we recall Sobolev embedding \cite[Theorem 2.4.4]{Kes} in our context.
\begin{Lemma}[Sobolev embedding] \label{lemPR:SobEmb}
	Let $\Omega$ be a bounded domain in $\mathbb{R}^d$ of class $C^1$  with $d\in \mathbb{N}.$ Then, we have the following continuous embedding with constant $C_s$:
	\begin{itemize}
		\item[$(a)$] $H^1(\Omega)\hookrightarrow L^p(\Omega)$ for $d>2,$ where $p=\frac{2d}{d-2},$
		\item[$(b)$] $H^1(\Omega) \hookrightarrow L^q(\Omega)$ for all $d=2\le q<\infty.$ 
	\end{itemize}
\end{Lemma}

In the next theorem, we recall the important inequality, known as the Gagliardo-Nirenberg inequality in the case of bounded domain with smooth boundary.

\begin{Theorem}[Gagliardo-Nirenberg inequality {\normalfont\cite{Nirenberg1959}}] \label{thm:G-NinBdd}
	Let $\Omega \subset \mathbb{R}^d,$ $d\in\{2,3\},$ be a convex domain or domain with $C^2$-boundary $\Gamma$. For any \( u \in W^{m,q}(\Omega) \), and for any integers \( j \) and \( m \) satisfying \( 0 \leq j < m \), the following inequality holds:
	\[
	\| D^j u \|_{L^p(\Omega)} \leq C \|  u \|_{W^{m,q}(\Omega)}^\theta \| u \|_{L^r(\Omega)}^{1-\theta} ,
	\]
	where \( D^j u \) denotes the \( j \)-th order weak derivative of \( u \), and:
	\[
	\frac{1}{p} = \frac{j}{d}+\theta \left( \frac{1}{q} - \frac{m}{d} \right) +  \frac{(1-\theta)}{r},
	\]
	for some constant \( C \) that depends on the domain \( \Omega \) but not on \( u \), where \( \theta \in [0,1] \).
\end{Theorem}

\begin{Lemma}[General Gronwall inequality \cite{Canon99}] \label{lem:Gronwall} 
	Let $f,g, h$ and $y$ be four locally integrable positive functions on $[t_0,\infty)$ such that
	\begin{align*}
		 y(t)+\int_{t_0}^t f(s)ds \le C+ \int_{t_0}^t h(s)ds + \int_{t_0}^t g(s)y(s)ds\  \text{ for all }t\ge t_0,
	\end{align*}
	where $C\ge 0$ is any constant. Then 
	\begin{align*}
		 y(t)+\int_{t_0}^t f(s)ds \le \left(C+\int_{t_0}^t h(s)ds\right) \exp\left( \int_{t_0}^t g(s)ds\right) \text{ for all }t\ge t_0.
	\end{align*}
\end{Lemma}

Now, we state a version of nonlinear generalization of Gronwalls' inequality that will be useful in our later analysis.
\begin{Theorem}{(A nonlinear generalization of Gronwall’s inequality \cite[Theorem 21]{DragSS}).}\label{lem:NonlinGronwall}
Let $\zeta(t)$ be a non-negative locally integrable function that satisfies the integral inequality
\begin{align*}
	\zeta(t)\le c+ \int_{t_0}^t \left(  a(s)\zeta(s) + b(s)\zeta^{\varepsilon}(s) \right) ds, c\ge 0, \ 0\le \varepsilon<1,
\end{align*}
where $a(t)$ and $b(t)$ are continuous non-negative functions for $t\ge t_0.$ Then, the following inequality holds:
\begin{align*}
	\zeta(t) \le \left\lbrace c^{1-\varepsilon} \exp{\left[(1-\varepsilon)\int_{t_0}^ta(s)ds\right]} + (1-\varepsilon)\int_{t_0}^t b(s) \exp{\left[(1-\varepsilon)\int_{s}^t a(r)dr \right]} ds\right\rbrace^{\frac{1}{1-\varepsilon}}.
\end{align*}
\end{Theorem}

\section{Error estimates for stationary viscous Burgers equation} \label{sec:StVB}
The main objective of this section is to establish error estimates for stationary  viscous Burgers' equation using PINNs. 
Let $\Omega \subset \mathbb{R}^d,$ $d\in\{2,3\},$ be a convex domain or domain with $C^2$-boundary $\Gamma$. We first consider the following stationary viscous Burgers equation: 
\begin{equation} \label{eqn:St-Burg}
	\left\{
	\begin{aligned}
		& -\nu \Delta u_{\infty} (x)+\sum_{i=1}^d u_{\infty} (x)\frac{\partial u_{\infty}}{\partial x_i} (x) =f_{\infty}(x)\  \text{ in }\  \Omega, \\
		& u_{\infty}(x)  = 0 \ \text{ on }\ \Gamma,
	\end{aligned} \right.
\end{equation}
where $\nu>0$ is a given constant and $f_{\infty}$ is the forcing term. The explicit function space for $f_{\infty}$  is stated in the following theorem. Now, we state a result on the existence and uniqueness  of weak as well as strong solutions of \eqref{eqn:St-Burg}.

\begin{Theorem}[{\n\cite{MTMAK21}}]\label{thm:Stat-Burg-WP}
    For given $f\in H^{-1}(\Omega),$ there exists a weak solution (of \eqref{eqn:St-Burg}) $u_{\infty}\in H^1_0(\Omega)$ such that 
	\begin{align*}
		\|u_{\infty}\|_{H^1(\Omega)} \le \frac{1}{\nu} \|f_{\infty}\|_{H^{-1}(\Omega)}.
	\end{align*}
		If $f_{\infty}\in L^2(\Omega),$ then $u_{\infty}\in H^2(\Omega) \cap H^1_0(\Omega)$ and satisfies
	\begin{align*}
		\|u_{\infty}\|_{H^2(\Omega)} \le C(\nu) \|f_{\infty}\|.
	\end{align*} 
	Moreover, if 
	\begin{align}\label{5p2}
		\nu>\left\{\begin{array}{cc}\left(\frac{2}{\lambda_1}\right)^{1/4}\|f_{\infty}\|_{H^{-1}(\Omega)}&\ \text{ if } d=2,\\
			\left(\frac{\sqrt{2}}{\lambda_1^{1/8}}\right)\|f_{\infty}\|_{H^{-1}(\Omega)}&\ \text{ if } d=3,
		\end{array}\right.
	\end{align}
	where $\lambda_1$ is the first eigenvalue of Dirichlet's Laplacian, then the weak solution is unique.

\end{Theorem}
By using Brouwer's fixed-point theorem (\cite[Chapter 2, Section 1]{TemamNSE01}) and a Galerkin approximation method, one  can establish the proof of above theorem and is standard (for example, see \cite[Theorems 2.1 - 2.3]{AKMTMRRB}), hence we omit here. 

\begin{Remark}
	The condition \eqref{5p2} means that  one has to take $\nu$ sufficiently large for arbitrary external forcing $f_{\infty}\in H^{-1}(\Omega)$ and small external forcing for arbitrary $\nu>0$. 
\end{Remark}

Now, we briefly discuss the existence of a strong solution for the following  stationary viscous Burgers equation with non-homogeneous boundary data, which is helpful in the sequel: 

\begin{equation} \label{eqn:St-BurgNH}
	\left\{
	\begin{aligned}
	&	-\nu \Delta u_{\infty} (x)+\sum_{i=1}^d u_{\infty} (x)\frac{\partial u_{\infty}}{\partial x_i}  (x)=f_{\infty} (x)\ \text{ in }\  \Omega, \\
	&	u_{\infty} (x) = g_\infty(x)\  \text{ on }\ \Gamma.
	\end{aligned} \right.
\end{equation}

\begin{Theorem}\label{thm-non-hom}
		For a given $f_\infty \in H^{-1}(\Omega)$ and $g_\infty \in H^{\frac{1}{2}}(\Gamma),$ there exists a weak solution $u_\infty \in H^1(\Omega)$ of \eqref{eqn:St-BurgNH} satisfying 
		$$\|u_\infty\|_{H^1(\Omega)} \le C \left(\|f_\infty\|_{H^{-1}(\Omega)} + \|g_\infty\|_{H^{\frac{1}{2}}(\Gamma)}\right),$$
		for some $C>0.$ 
		
	For $f_\infty \in L^2(\Omega)$ and $g_\infty \in H^{\frac{3}{2}}(\Gamma),$  every weak solution of \eqref{eqn:St-BurgNH} is  a  strong solution $u_\infty \in H^2(\Omega)$  satisfying 
	$$\|u_\infty\|_{H^2(\Omega)} \le C \left(\|f_\infty\| + \|g_\infty\|_{H^{\frac{3}{2}}(\Gamma)}\right),$$
	for some $C>0.$ Moreover, if $\nu>0$ is sufficiently large, then the solution is unique. 
\end{Theorem}

\begin{proof}
	For a given $g_\infty\in H^{s+\frac{1}{2}}(\Gamma),$ $s\in\{0,1\}$, one can define a lifting operator $\ell: H^{s+\frac{1}{2}}(\Gamma) \rightarrow H^{1+s}(\Omega)$ and let $z:=\ell(g_\infty).$ We also have 
	\begin{align}\label{eqn-lift}
		\|z\|_{H^{1+s}(\Omega)}\leq C_{\ell}\|g\|_{H^{s+\frac{1}{2}}(\Gamma)}, \ s\in\{0,1\}. 
	\end{align}
	Then $v:=u_\infty-z$ satisfies 
	\begin{equation}  \label{eq:St-BrgNH-22}
		\left\{
		\begin{aligned}
			& -\nu \Delta v + \sum_{i=1}^d v \frac{\partial v}{\partial x_i} + \sum_{i=1}^d v \frac{\partial z}{\partial x_i} + \sum_{i=1}^d \frac{\partial v}{\partial x_i} z  =f_{\infty} +\nu\Delta z-\sum_{i=1}^d z \frac{\partial z}{\partial x_i} \text{ in } \Omega, \\
			& v  = 0 \text{ on }\Gamma.
		\end{aligned} \right.
	\end{equation}
For $g_\infty\in H^{\frac{1}{2}}(\Gamma)$ and $f_\infty \in H^{-1}(\Omega)$,	one can show that 
\begin{align*}
	\bigg\|f_{\infty} +\nu\Delta z-\frac{1}{2}\sum_{i=1}^d  \frac{\partial z^2}{\partial x_i}\bigg\|_{H^{-1}(\Omega)}&\leq \|f_{\infty} \|_{H^{-1}(\Omega)}+\nu\|z\|_{H^1(\Omega)}+C\|z\|_{L^4(\Omega)}^2\\&\leq  \|f_{\infty} \|_{H^{-1}(\Omega)}+C_{\ell}\|g\|_{H^{\frac{1}{2}}(\Gamma)}+C_s\|z\|_{H^1(\Omega)}^2\\&\leq   \|f_{\infty} \|_{H^{-1}(\Omega)}+C\|g\|_{H^{\frac{1}{2}}(\Gamma)}+C\|g\|_{H^{\frac{1}{2}}(\Gamma)}^2,
\end{align*}
where we have used Sobolev's inequality and \eqref{eqn-lift}. Then by using similar arguments as in the proof of \cite[Theorems 2.22 - 2.2]{AKMTMRRB} (see Theorem \ref{thm:Stat-Burg-WP}), one can obtain the existence of a weak solution to the problem \eqref{eq:St-BrgNH-22}. 

For $g_\infty\in H^{\frac{3}{2}}(\Gamma)$ and $f_\infty \in L^2(\Omega)$,	it is immediate that  
	\begin{align*}
		\bigg\|f_{\infty} +\nu\Delta z-\sum_{i=1}^d z \frac{\partial z}{\partial x_i}\bigg\| & \le \|f_{\infty}\| + \nu \|z\|_{H^2(\Omega)}+ \|z\|_{L^4(\Omega)} \|\nabla z\|_{L^4(\Omega)}\\
		& \le \|f_{\infty}\| + \nu C_\ell \|g_\infty\|_{H^{\frac{3}{2}}(\Gamma)} + C_sC_\ell \|g_\infty\|_{H^{\frac{3}{2}}(\Gamma)}^2,
	\end{align*}
where once again we have used Sobolev's inequality and \eqref{eqn-lift}. Now, following a similar argument as in the proof of \cite[Theorem 2.3]{AKMTMRRB}   (Theorem \ref{eqn:St-Burg}), one can show that the weak solution is strong and $v\in H^2(\Omega)\cap H^1_0(\Omega)$ of \eqref{eq:St-BrgNH-22}. Therefore, since $u_{\infty}=v+z$, the problem \eqref{eqn:St-BurgNH} admits unique weak solution for sufficiently large $\nu$,  $f_\infty \in H^{-1}(\Omega)$ and $g_\infty \in H^{\frac{1}{2}}(\Gamma),$ and a strong solution for $f_\infty \in L^2(\Omega), g_\infty \in H^{\frac{3}{2}}(\Gamma)$.
\end{proof}

\subsection{Minimization problem}

To discuss the error estimates for stationary viscous Burgers equation using PINNs, it is convenient to consider the following minimization problem (\cite{BiswasDNN-NSE}):
\begin{align*}
	\inf_{u_\infty \in \text{appropriate space}} \bigg\{\Big\|- \nu \Delta u_\infty  + \sum_{i=1}^d u_\infty \frac{\partial u_\infty}{\partial x_i} - f_\infty \Big\|^2 + \| u_\infty|_\Gamma\|^2_{L^2(\Gamma)}\bigg\},
\end{align*}
and it has a unique solution, with the value of the infimum being zero and the infimum is attained at the solution $u_\infty$ of \eqref{eqn:St-Burg}. Thus, in order to approximate $u_\infty$ using a DNN, one considers the following  loss function:
\begin{align*}
	\Big\|- \nu \Delta u_\infty^N  + \sum_{i=1}^d u_\infty^N \frac{\partial u_\infty^N}{\partial x_i} - f_\infty \Big\|^2 + & \| u_\infty^N|_\Gamma\|^2_{L^{2}(\Gamma)} ,\ u_\infty^N\in\mathcal{F}_N,
\end{align*}
under the restriction that $u_\infty^N$ is bounded in $ H^2(\Omega)$. More precisely, let $p_k\in \Omega$ for $k=1,2,\ldots,m$ and $q_j \in \Gamma$ for $j=1,2,\ldots,n$ be some chosen points. Then on these points, we write the cost functional as
\begin{align*}
	\inf_{\{\tilde{u}^N\in \mathcal{F}_N, \ \|\tilde{u}^N\|_{H^2(\Omega)}\le M\}} \sum_{k=1}^m \bigg| - \nu \Delta \tilde{u}^N(p_k)  + \sum_{i=1}^d \tilde{u}^N(p_k) \frac{\partial \tilde{u}^N}{\partial x_i}(p_k) - f_\infty(p_k)\bigg|^2 + \sum_{j=1}^n |\tilde{u}^N(q_j)|^2.
\end{align*}

\subsection{Error estimates}
In the next theorem, we show that if there is a DNN $u_\infty^N$ that approximates \eqref{eqn:St-Burg}, then it is ``close enough'' to $u_\infty,$ the exact solution of \eqref{eqn:St-Burg}.

\begin{Theorem} \label{th;error-statVB}
	Let $f_{\infty}\in L^2(\Omega)$ and $u_\infty$ be the solution of stationary Burgers equation \eqref{eqn:St-Burg}, and for a given $\varepsilon>0$, there exists $u_\infty^N\in \mathcal{F}_N$ with $\|u_{\infty}^N\|_{H^2(\Omega)}\leq M$ such that 
	\begin{align} \label{eqn:stBurg_givenHypo}
		\Big\|-\nu \Delta u_\infty^N + \sum_{i=1}^d u_\infty^N \frac{\partial u_\infty^N}{\partial x_i}-f_\infty\Big\|^2 + \|\left.u_\infty^N\right|_\Gamma \|_{L^2(\Gamma)}^2 \le \varepsilon^2.
	\end{align}
	Then for 
	\begin{align}
		\nu> \max\left\{\sqrt{3}\left(\frac{1}{2\lambda_1}\right)^{1/4}\|f_{\infty}\|_{H^{-1}(\Omega)}^{1/2},\frac{3}{2}C_{\ell}M\varepsilon^{2/3}\right\}
	\end{align}
	 and for some $C>0,$ the following estimate holds:
	\begin{align*}
		\|u_\infty- u_\infty^N\|_{H^1(\Omega)} \le C \varepsilon^{\frac{2}{3}}.
	\end{align*}
	
	Furthermore, under the assumption 
	\begin{align} \label{eqn:stBurg_givenImpHypo}
		\Big\|-\nu \Delta u_\infty^N + \sum_{i=1}^d u_\infty^N \frac{\partial u_\infty^N}{\partial x_i}-f_\infty\Big\|^2 + \|\left.u_\infty^N\right|_\Gamma \|_{H^{1/2}(\Gamma)}^2 \le \varepsilon^2,
	\end{align}
	we obtain the improved estimate as 
	\begin{align*}
		\|u_\infty- u_\infty^N\|_{H^1(\Omega)}\le \varepsilon.
	\end{align*}
\end{Theorem}

\begin{proof}
	Let $\displaystyle -\wt{f}_\infty:=-\nu \Delta u_\infty^N + \sum_{i=1}^d u_\infty^N \frac{\partial u_\infty^N}{\partial x_i}-f_\infty$ and $-\wt{g}_\infty:= \left.u_\infty^N\right|_\Gamma .$ Then $u_\infty^N$ satisfies
	\begin{equation*}
		\left\{
		\begin{aligned} 
		&	-\nu \Delta u_\infty^N + \sum_{i=1}^d u_\infty^N \frac{\partial u_\infty^N}{\partial x_i}  = f_\infty-\wt{f}_\infty \text{ in }\Omega, \\
		&	u_\infty^N  = - \wt{g}_\infty \text{ on }\Gamma. 
		\end{aligned} \right.
	\end{equation*}
	Set $w=u_\infty-u_\infty^N$ to obtain
	\begin{equation*}
		\left\{
		\begin{aligned}
		&	-\nu \Delta w + \sum_{i=1}^d (u_\infty-u_\infty^N) \frac{\partial u_\infty}{\partial x_i} + \sum_{i=1}^d u_\infty^N \frac{\partial }{\partial x_i}(u_\infty - u_\infty^N)  = \wt{f}_\infty \text{ in }\Omega, \\
		&	w  = \wt{g}_\infty \text{ on }\Gamma,
		\end{aligned} \right.
	\end{equation*}
	which further can be re-written as
	\begin{equation} \label{eq:StBg w=u-uN}
		\left\{
		\begin{aligned}
		&	-\nu \Delta w + \sum_{i=1}^d w \frac{\partial u_\infty}{\partial x_i} + \sum_{i=1}^d u_\infty \frac{\partial w}{\partial x_i} -  \sum_{i=1}^d w \frac{\partial w}{\partial x_i}  = \wt{f}_\infty \text{ in }\Omega, \\
		&	w  = \wt{g}_\infty \text{ on }\Gamma  .
		\end{aligned} \right.
	\end{equation}
The solvability of the above problem is discussed in Theorem \ref{thm-non-hom}. 	Consider the lifting operator $\ell: H^{1/2}(\Gamma)\rightarrow H^1(\Omega)$ and denote $z:=\ell(\wt{g}_\infty).$ Then $ z$ and $\wt{g}_\infty$ satisfy (\cite{LioMag-V2,McLean-2000})
	\begin{align} \label{eqn:z_H1 le g_H1/2}
		\|z\|_{H^1(\Omega)} \le C_\ell \|\wt{g}_\infty\|_{H^{1/2}(\Gamma)}.
	\end{align}
	Now, set $v=w-z,$ so that we have
	\begin{equation*}
		\left\{
		\begin{aligned}
		&	-\nu \Delta (v+z) + \sum_{i=1}^d (v+z) \frac{\partial u_\infty}{\partial x_i} + \sum_{i=1}^d u_\infty \frac{\partial }{\partial x_i}(v+z) -  \sum_{i=1}^d (v+z) \frac{\partial }{\partial x_i}(v+z)   = \wt{f}_\infty\text{ in }\Omega, \\
			& v  = 0 \text{ on }\Gamma  .
		\end{aligned} \right.
	\end{equation*}
	On simplifying this, we obtain
	\begin{equation*}
		\left\{
		\begin{aligned}
		&	-\nu \Delta v + \sum_{i=1}^d v \frac{\partial u_\infty}{\partial x_i}  +\sum_{i=1}^d z \frac{\partial u_\infty}{\partial x_i} + \sum_{i=1}^d u_\infty \frac{\partial v}{\partial x_i}  + \sum_{i=1}^d u_\infty \frac{\partial z}{\partial x_i}   \\
		& \qquad \qquad 	-  \sum_{i=1}^d v \frac{\partial v}{\partial x_i}- \sum_{i=1}^d v \frac{\partial z}{\partial x_i} - \sum_{i=1}^d z \frac{\partial v}{\partial x_i}  = \wt{f}_\infty + \nu \Delta z +  \sum_{i=1}^d z \frac{\partial z}{\partial x_i} \text{ in }\Omega, \\
		&	v  = 0 \text{ on }\Gamma  .
		\end{aligned} \right.
	\end{equation*}
	Taking the  inner product with $v$ leads to
	\begin{equation}  \label{eqn:Stat-Burg-Smpl-WF} 
	\begin{aligned}
		\nu \|\nabla v\|^2 & = -\Big(\sum_{i=1}^d v \frac{\partial u_\infty}{\partial x_i}, v\Big) - \Big( \sum_{i=1}^d u_\infty \frac{\partial v}{\partial x_i}, v\Big) - \Big(\sum_{i=1}^d z \frac{\partial u_\infty}{\partial x_i},v\Big) - \Big(\sum_{i=1}^d u_\infty \frac{\partial z}{\partial x_i}, v\Big) \\
		& \quad + \Big(\sum_{i=1}^d v \frac{\partial z}{\partial x_i},v\Big) + \Big( \sum_{i=1}^d z \frac{\partial v}{\partial x_i}, v\Big) +(\wt{f}_\infty,v) - \nu (\nabla z, \nabla v) + \Big(\sum_{i=1}^d z \frac{\partial z}{\partial x_i},v\Big)\\
		& = \Big(\sum_{i=1}^d u_\infty \frac{\partial v}{\partial x_i}, v\Big)+ \Big(\sum_{i=1}^d \frac{\partial v}{\partial x_i}, u_\infty z\Big) + \frac{1}{2}\Big(\sum_{i=1}^d v \frac{\partial z}{\partial x_i},v\Big) - \frac{1}{2}\Big(\sum_{i=1}^d z \frac{\partial v}{\partial x_i},z\Big) \\
		& \qquad  +(\wt{f}_\infty,v) - \nu (\nabla z,\nabla v)
		\\&=\sum_{k=1}^6I_k,
	\end{aligned}
	\end{equation}
where we have performed integration by parts. 	
	Using H\"{o}lder's, Ladyzhenskaya's and Poincar\'e's inequalities, and Theorem \ref{thm:Stat-Burg-WP}, we estimate $I_1$ as
	\begin{align*}
	|I_1|&=	\Big\vert \Big(\frac{1}{2}\sum_{i=1}^d v \frac{\partial u_\infty}{\partial x_i} , v\Big)  \Big\vert   \le \frac{1}{2}\|\nabla u_\infty\| \|v\|_{L^4(\Omega)}^2\leq \frac{1}{\sqrt{2}} \|\nabla u_\infty\| \|\nabla v\|\|v\| \leq\sqrt{\frac{1}{2\lambda_1}} \|\nabla u_\infty\| \|\nabla v\|^2\\&\le\frac{1}{\nu} \sqrt{\frac{1}{2\lambda_1}} \|f_\infty\|_{H^{-1}(\Omega)}\|\nabla v\|^2.
	\end{align*}
	 A generalized H\"{o}lder's inequality,  \eqref{eqn:z_H1 le g_H1/2}, and Theorem \ref{thm:Stat-Burg-WP} lead to
	\begin{align*}
	|I_2|&=	\Big\vert \Big(\sum_{i=1}^d  \frac{\partial v}{\partial x_i}, u_\infty z\Big)  \Big\vert \le \|\nabla v\| \|u_\infty\|_{L^4(\Omega)} \|z\|_{L^4(\Omega)}  \le \frac{\eta\nu}{4}\|\nabla v\|^2 + C \|z\|_{H^1(\Omega)}^2\|u_\infty\|_{H^1(\Omega)}^2 \\
		& \le \frac{\eta\nu}{4}\|\nabla v\|^2 + C \|f_\infty\|_{H^{-1}(\Omega)}^2 \|\wt{g}_\infty\|_{H^{1/2}(\Gamma)}^2,
	\end{align*}
	where the constant $C$ depends on $C_s, \nu,$ and $C_\ell,$ and the constant $\eta>0$ which will be specified later. 
	Using generalized H\"{o}lder's inequality, Sobolev's embedding and \eqref{eqn:z_H1 le g_H1/2}, we obtain
	\begin{align*}
	|I_3|=	\Big\vert \frac{1}{2} \Big(\sum_{i=1}^d v \frac{\partial z}{\partial x_i},v\Big) \Big\vert \le  \frac{1}{2}\|\nabla z\| \|v\|_{L^4(\Omega)}^2 \le  \frac{1}{2}\|z\|_{H^1(\Omega)} \|\nabla v\|^2 \le \frac{C_\ell}{2} \|\wt{g}_\infty\|_{H^{1/2}(\Gamma)} \|\nabla v\|^2.
	\end{align*}
	Once again applying generalized H\"{o}lder's and Cauchy Schwarz inequalities, and \eqref{eqn:z_H1 le g_H1/2}, we deduce
	\begin{align*}
	|I_4|=	\Big\vert   \frac{1}{2}\Big(\sum_{i=1}^d z \frac{\partial v}{\partial x_i},z\Big) \Big\vert \le \frac{1}{2} \|\nabla v\| \|z\|_{L^4(\Omega)}^2 \le \frac{\eta\nu}{4}\|\nabla v\|^2 + C \|\wt{g}_\infty\|_{H^{1/2}(\Gamma)}^4,
	\end{align*}
	where the constant $C$ depends on $C_s, \nu,$ and $C_\ell.$
	H\"older's and Young's inequalities yield
	\begin{align*}
|I_5|&	= |(\wt{f}_\infty,v)| \le \frac{\eta\nu}{4}\|\nabla v\|^2 + C_\nu\|\wt{f}_\infty\|^2_{H^{-1}(\Omega)},\\
|I_6|&=|\nu (\nabla z, \nabla v)| \le \frac{\eta\nu}{4} \|\nabla v\|^2 +C_{\nu} \|\nabla z\|^2 \le \frac{\eta\nu}{4} \|\nabla v\|^2 +C \|\wt{g}_\infty\|_{H^{1/2}(\Gamma)}^2.
\end{align*}

	For the given hypothesis \eqref{eqn:stBurg_givenHypo}, we use interpolation and Trace theories to find
	\begin{align*}
		\|\wt{g}_\infty\|_{H^{\frac{1}{2}}(\Gamma)} \le \|\wt{g}_\infty\|_{L^2(\Gamma)}^{\frac{2}{3}} \|\wt{g}_\infty\|_{H^{\frac{3}{2}}(\Gamma)}^{\frac{1}{3}} \le \varepsilon^{\frac{2}{3}}\|u_\infty^N\|_{H^2(\Omega)}^{\frac{1}{3}} \le M \varepsilon^{\frac{2}{3}}.
	\end{align*}
	Moreover, using \eqref{eqn:stBurg_givenHypo}, we also have 
	\begin{align*}
		\|\wt{f}_\infty\|_{H^{-1}(\Omega)}\leq \varepsilon. 
	\end{align*}
	Let $\eta=\frac{1}{3}.$ Now, one can choose $\nu$ sufficiently large or $\|f_\infty\|_{H^{-1}(\Omega)}$ sufficiently small so that 
	\begin{align*}
		\frac{\nu}{3} > \frac{1}{\nu} \sqrt{\frac{1}{2\lambda_1}}\|f_\infty\|_{H^{-1}(\Omega)}, \text{ that is, } \nu^2 \ge 3 \sqrt{\frac{1}{2\lambda_1}}\|f_\infty\|_{H^{-1}(\Omega)}.
	\end{align*}
	Since, $\|\wt{g}_\infty\|_{H^{\frac{1}{2}}(\Gamma)} \le M \varepsilon,$ one can choose $\varepsilon>0$ sufficiently small so that 
	\begin{align*}
		\frac{C_\ell C_2}{2}\varepsilon^{\frac{2}{3}}< \frac{\nu}{3}.
	\end{align*}
	Combining the above estimates in \eqref{eqn:Stat-Burg-Smpl-WF}, we obtain 
	\begin{align*}
		\frac{\nu}{3}\|\nabla v\|^2 \le C_\nu (\|f_\infty\|_{H^{-1}(\Omega)}^2 + \varepsilon^{\frac{4}{3}}) \varepsilon^{\frac{4}{3}},  
	\end{align*}
	and hence using Poincar\'e's inequality, we deduce
	\begin{align*}
		\| v\|_{H^1(\Omega)}^2 \le C\varepsilon^{\frac{4}{3}},  
	\end{align*}
	for some $C=C(\nu,\lambda_1, \|f_\infty\|_{H^{-1}(\Omega)}, C_s)>0.$
	Recall that $u_\infty -u_\infty^N= w= v+z$ and $\|z\|_{H^1(\Omega)} \le C_\ell \|\wt{g}_\infty\|_{H^{1/2}(\Gamma)} \le C_\ell M \varepsilon^{\frac{2}{3}}$ imply
	\begin{align}
		\| u_\infty - u_\infty^N\|_{H^1(\Omega)} \le C   \varepsilon^{\frac{2}{3}}, 
	\end{align}
	for some $C=C(\nu,\lambda_1, \|f_\infty\|_{H^{-1}(\Omega)}, C_s)>0.$
	Now, to obtain the improved estimate using \eqref{eqn:stBurg_givenImpHypo}, we directly use $$\|z\|_{H^1(\Omega)} \le C \|\wt{g}_\infty\|_{H^{\frac{1}{2}}(\Gamma)} \le C \varepsilon,$$ by combining the estimates of $I_1,\ldots,I_6$. For the sake of repetitions, here we skip it. 
\end{proof}

\begin{Remark}
	An analogous result can be obtained for the stationary viscous Burgers equation with non-homogenous boundary data just by mimicking the above proof.
\end{Remark}

In the next theorem, we show that there exists a DNN $u_\infty^N \in \mathcal{F}_N$ that is ``close'' to the solution $u_\infty$ of \eqref{eqn:St-Burg} as well as it also approximates the corresponding cost functional.

\begin{Theorem} \label{th:existDNNerrorStVB}
	 For a given $f_\infty\in L^2(\Omega),$ let $u_\infty$ be the strong solution of \eqref{eqn:St-Burg}.  Then for arbitrary  $\varepsilon>0$, there exists a DNN $u_\infty^N \in \mathcal{F}_N$ such that 
	\begin{align*}
		\|u_\infty - u_\infty^N\|_{H^2(\Omega)} \le \varepsilon
	\end{align*}
	and
	\begin{align}\label{eqn-error}
		\Big\|-\nu \Delta u_\infty^N + \sum_{i=1}^d u_\infty^N \frac{\partial u_\infty^N}{\partial x_i} - f_\infty \Big\|^2 + \|\left.u_\infty^N\right|_\Gamma \|_{L^2(\Gamma)}^2 \le C\varepsilon^2,
	\end{align}
	for some $C=C(\nu, \|f_\infty\|, C_{Tr} )>0.$
\end{Theorem}

\begin{proof}
Using Theorem \ref{thm:Stat-Burg-WP}, we have that that for a given $f_\infty\in L^2(\Omega),$ the strong solution, $u_\infty$, of \eqref{eqn:St-Burg} belongs to $H^2(\Omega)\cap H^1_0(\Omega).$ Our first aim is to find a DNN $u_\infty^N \in \mathcal{F}_N$ for a suitable $N$ such that 
\begin{align*}
	\|u_\infty - u_\infty^N\|_{H^2(\Omega)} \le \varepsilon.
\end{align*}
To achieve this, recall that $C^\infty_c(\Omega)$ is dense in $ H^{2}(\Omega) \cap H^1_0(\Omega)$. Therefore for the given $\varepsilon>0,$ there exists $g\in C^\infty_c(\Omega)$ such that $\| u_\infty-g\|_{H^2(\Omega)} \le \frac{\varepsilon}{2}.$ 
		
Since $g\in C^\infty_c(\Omega),$ there exists a compact set $\mathcal{K}\subset \Omega$ such that
\begin{align*}
\text{supp}(g) \subset \mathcal{K}, \, D^\alpha g=0 \text{ in }\Omega\backslash \mathcal{K} \ \text{ for all multiindices }\ \alpha. 
\end{align*}
Let us fix $K=\overline{\Omega}$. Note that $\mathcal{K}\subset\Omega\subset K$ and since, $g\in C^k(K)$ for all $k=1,2,\ldots,$ we can apply \cite[Theorem 3.3]{XieDNN_EA_CAMWA} for $g.$ For any multi-nidex $\beta$ with $|\beta|=2,$ $D^\beta g$ is a Lipschitz function on $K,$ and hence we have 
        \begin{align} \label{eq:lip}
            | D^\beta g(x) - D^\beta g(y)| \le L|x-y| \ \text{ for all }\  x,y \in K,
        \end{align}
for some constant $L>0$. Let $\alpha$ be any multiindex such that $|\alpha|\le 2.$ Then from \cite[Theorem 3.3]{XieDNN_EA_CAMWA}, we can find a $u_\infty^N$ represented by a DNN such that 
\begin{align*}
    \sup_{x\in K}|D^{\alpha}g (x) -D^\alpha u_\infty^N(x)| &  \le C \frac{1}{N^{ \frac{2-|\alpha|}{d} }} \,  \sup_{x,y \in K, \ |x-y|\le \frac{1}{N^{1/d}}} |D^\beta g(x) - D^\beta g(y)| \\
    & \le C \frac{1}{N^{ \frac{2-|\alpha|}{d} }} \, \sup_{ x,y \in K, \ |x-y|\le \frac{1}{N^{1/d}}} L |x-y| \le CL \frac{1}{N^{\frac{3-|\alpha|}{d}}},
\end{align*}
where in the second last inequality, we have  used \eqref{eq:lip}. 
We further deduce 
\begin{align*}
    \int_K \vert (D^{\alpha}g (x) - D^\alpha u_\infty^N(x)) \vert^2 dx \le \sup_{x\in K}|D^{\alpha}g (x) -D^\alpha u_\infty^N(x)|^2 |\Omega| \le (CL)^2 \frac{1}{N^{\frac{2(3-|\alpha|)}{d}}} |\Omega|.   
\end{align*}
Now, choosing $N$ sufficiently large such that 
\begin{align*}
    N^{\frac{3-|\alpha|}{d}} \ge \frac{2CL\sqrt{ |\Omega|}}{\varepsilon} \text{ for all multi-indices } \alpha \text{ such that }|\alpha|\le 2,  
\end{align*}
we obtain 
\begin{align*}
\|g-u_\infty^N\|_{H^2(K)} \le \frac{\varepsilon}{2}.
\end{align*}
It is sufficient to choose $N$ such that $N \ge \left( \frac{2CL\sqrt{ |\Omega|}}{\varepsilon} \right)^d.$
Therefore, we finally have 
\begin{align*}
\|u_\infty-u_\infty^N\|_{H^2(\Omega)} \le  \|u_\infty -g\|_{H^2(\Omega)} +  \|g-u_\infty^N\|_{H^2(\Omega)} \le \varepsilon,
\end{align*}

Also, note that 
	\begin{align*}
		\| u_\infty^N\|_{H^2(\Omega)} \le \|u_\infty - u_\infty^N\|_{H^2(\Omega)} + \|u_\infty\|_{H^2(\Omega)}  \le \varepsilon + \|u_\infty\|_{H^2(\Omega)} \le \varepsilon + C\|f_\infty\|,
	\end{align*} 
by using Theorem \ref{thm:Stat-Burg-WP}.	Now, using H\"older's inequality after replacing $f_\infty$ in the left hand side of \eqref{eqn-error} by  \eqref{eqn:St-Burg}, we obtain
	\begin{align*}
		\Big\|&-\nu\Delta u_\infty^N +  \sum_{i=1}^d u_\infty^N \frac{\partial u_\infty^N}{\partial x_i} - f_\infty \Big\|  = \Big\|-\nu \Delta u_\infty^N + \sum_{i=1}^d u_\infty^N \frac{\partial u_\infty^N}{\partial x_i} + \nu\Delta u_\infty - \sum_{i=1}^d u_\infty \frac{\partial u_\infty}{\partial x_i} \Big\| \\
		& \le \nu \| \Delta (u_\infty - u_\infty^N)\| + \Big\| \sum_{i=1}^d (u_\infty - u_\infty^N) \frac{\partial u_\infty}{\partial x_i}\Big\|  + \Big\| \sum_{i=1}^d  u_\infty^N \frac{\partial (u_\infty - u_\infty^N)}{\partial x_i}\Big\|  \\
		& \le \nu \|u_\infty - u_\infty^N\|_{H^2(\Omega)} + \|u_\infty - u_\infty^N\|_{L^4(\Omega)} \|\nabla u_\infty\|_{L^4(\Omega)}   + \|u_\infty^N\|_{L^4(\Omega)}\|\nabla (u_\infty - u_\infty^N)\|_{L^4(\Omega)}  \\
		& \le \varepsilon \left(  \nu  + \|\nabla u_\infty\|_{L^4(\Omega)} + \|u_\infty^N\|_{L^4(\Omega)} \right)\\
		& \le \varepsilon \left( \nu + C\varepsilon +C \|f_\infty\|\right),
	\end{align*}
	where the constant $C$ depends on $C_s, $ and $\nu.$ 
	
Let us now take $-\wt{f}_\infty=-\nu \Delta u_\infty^N + \sum_{i=1}^d u_\infty^N \frac{\partial u_\infty^N}{\partial x_i}-f_\infty$ and $-\wt{g}_\infty= \left.u_\infty^N\right|_\Gamma .$ Then $w:=u_\infty-u_\infty^N$ satisfies \eqref{eq:StBg w=u-uN}, and  since $u_\infty|_\Gamma=0$, we have
\begin{align*}
    \|\left.u_\infty^N\right|_\Gamma\|_{L^2(\Gamma)} &\le \| \left.u_\infty^N\right|_\Gamma - \left.u_\infty\right|_\Gamma\|_{L^2(\Gamma)} \le C\| \left.(u_\infty^N- u_\infty)\right|_\Gamma\|_{H^{\frac{1}{2}}(\Gamma)} \\&\le C_{Tr}\|u_\infty - u_\infty^N\|_{H^1(\Omega)} \le C_{Tr} \varepsilon.
\end{align*}
Combining these two estimates, we obtain the required result \eqref{eqn-error}.
\end{proof}

\begin{Remark}[Activation functions]
	In this paper and hence in the above theorem, we use smooth activation functions, that is, functions $\phi \in C^\infty(\mathbb{R})$ with non-vanishing derivatives at certain points (such as the sigmoid ($\sigma(x) = \frac{e^x}{e^x + 1}$), or the hyperbolic tangent $\tanh$ ($\tanh(x)= \frac{e^x -e^{-x}}{e^x+e^{-x}})$). These are infinitely differentiable and non-polynomial, which ensures that neural networks can approximate not only functions but also their derivatives with good accuracy (see \cite{XieDNN_EA_CAMWA}).
\end{Remark}

\section{Solvability of non-stationary viscous Burgers equation} \label{sec:WellposedNStVB}
Let $\Omega \subset \mathbb{R}^d,$ $d\in\{2,3\},$ be a convex domain or domain with $C^2$-boundary $\Gamma,$ and $T>0$ be fixed. Consider the viscous Burgers equation
\begin{equation} \label{eq:Burg}
	\left\{
\begin{aligned}
& \frac{\partial u}{\partial t} (x,t)- \nu \Delta u(x,t) + \sum_{i=1}^d u(x,t) \frac{\partial u}{\partial x_i}(x,t)  = f(x,t) \text{ in } \Omega \times (0,T), \\
& u(x,t)=0 \text{ on }\Gamma\times (0,T), \\
& u(x,0)=u_0(x) \text{ in } \Omega,
\end{aligned} \right.
\end{equation}
where $f \in L^2(0,T; L^2(\Omega))$ is given forcing term, and $\nu>0$ is a given constant (viscosity coefficient). 

Denote
\begin{align}
B(u)= \sum_{i=1}^d u \frac{\partial u}{\partial x_i} =\xi u \cdot \nabla u = \frac{1}{2}\sum_{i=1}^d \frac{\partial }{ \partial x_i}(u^2) = \frac{1}{2}\nabla \cdot (\xi u^2),
\end{align}
where $\xi=(1,1,\ldots,1)^{\top} \in \mathbb{R}^d,$ and an unbounded operator
\begin{align}
A=-\Delta:  D(A)\subset L^2(\Omega) \rightarrow L^2(\Omega), \text{ with } D(A)=H^2(\Omega) \cap H^1_0(\Omega).
\end{align}
Then the operator $(A,D(A))$ generates an analytic semigroup $\{e^{tA}\}_{t\ge 0}$ on $L^2(\Omega)$ (see \cite[Theorem 12.40]{RROG}). The mild form of \eqref{eq:Burg} can be written as 
\begin{align}\label{1p4}
u(t) = e^{-\nu t A}u_0 - \int_0^t e^{-\nu (t-s)A} B (u(s))ds + \int_0^t e^{-\nu (t-s)A} f(s) ds\  \text{ for all } \ 0<t<T.
\end{align}
Our first aim is to show the existence of a strong solution to the problem \eqref{eq:Burg}. Even though, the unique solvability results for the problem \eqref{eq:Burg} is available in the literature (\cite{TemamNSE01}), for any $f \in L^2(0,T; L^2(\Omega))$ and $u_0\in H_0^1(\Omega),$ we provide a proof of the existence of strong solutions by a fixed point argument and using an approximation method.

\subsection{Local solution}
We first start  by recalling some results which are crucial to establish the well-posedness of the problem \eqref{eq:Burg}.
\begin{Lemma}[\cite{MTM21-LpSol,KatoLpNSE}]
The semigroup $e^{-\nu t A}: L^p(\Omega) \longrightarrow L^q(\Omega)$ is bounded for all $1<p\le q <\infty, \, t>0$ and the following estimates hold:
\begin{align*}
& (i) \quad \|e^{-\nu t A}g\|_{L^q(\Omega)} \le C t^{-\frac{d}{2} (\frac{1}{p} - \frac{1}{q})} \|g\|_{L^p(\Omega)},\\
& (ii) \quad \|\nabla e^{-\nu t A}g\|_{L^q(\Omega)} \le C t^{- \frac{1}{2} -\frac{d}{2} (\frac{1}{p} - \frac{1}{q})} \|g\|_{L^p(\Omega)},
\end{align*}
for all $t\in (0,T),$ $g\in L^p(\Omega)$ and for some $C=C(p,q)>0.$
\end{Lemma}
If we take $q=p$ and $p=p/2$ in the above lemma, we obtain
\begin{align}\label{3p1}
\| e^{-\nu t A}g\|_{L^p(\Omega)} \le C t^{-\frac{d}{2} (\frac{2}{p}-\frac{1}{p})}\|g\|_{L^{\frac{p}{2}}(\Omega)} =  Ct^{-\frac{d}{2p}} \|g\|_{L^{\frac{p}{2}}(\Omega)},
\end{align}
and
\begin{align}\label{3p2}
\| \nabla e^{-\nu t A}g\|_{L^p(\Omega)} \le C t^{-\frac{1}{2}-\frac{d}{2} (\frac{2}{p}-\frac{1}{p})}\|g\|_{L^{\frac{p}{2}}(\Omega)} =  Ct^{-\frac{1}{2}-\frac{d}{2p}} \|g\|_{L^{\frac{p}{2}}(\Omega)}.
\end{align}
Also, by using \eqref{3p1}-\eqref{3p2}, we have (see \cite[(1.8)]{MTM21-LpSol})
\begin{align}\label{3p2p5}
	\| e^{-\nu A t} \xi v\cdot \nabla u\|_{L^p(\Omega)} \le C t^{-\frac{1}{2} - \frac{d}{2p}}\|u\|_{L^p(\Omega)}\|v\|_{L^p(\Omega)} ,
\end{align}
or, in particular, we have 
\begin{align}\label{3p3}
\| e^{-\nu A t} B(u)\|_{L^p(\Omega)} = \| e^{-\nu A t} \frac{1}{2} \nabla \cdot (\xi u^2)\|_{L^p(\Omega)} \le C t^{-\frac{1}{2} - \frac{d}{2p}}\|u^2\|_{L^\frac{p}{2}(\Omega)} = C t^{-\frac{1}{2} - \frac{d}{2p}}\|u\|_{L^p(\Omega)}^2.
\end{align}
\begin{Definition}[Mild solution {\normalfont\cite[Definition 3.1, Ch. 3, Part II]{BDDM}}]
	For given $u_0\in L^p(\Omega)$ and $f\in L^\infty(0,T; L^2(\Omega)),$ for $2\leq p<\infty$, a function $u\in C([0,T]; L^p(\Omega))$ is called a mild solution of the problem \eqref{eq:Burg} if it satisfies \eqref{1p4}.
\end{Definition}
\begin{Theorem}[Local existence]\label{mild-solution}
	For any $p>d$, $u_0\in L^p(\Omega)$ and $f\in L^{\infty}(0,T;L^2(\Omega))$, there exists a unique local mild solution of the problem \eqref{eq:Burg}, that is, there exists a time $0<T^*<T$ such that $u\in C([0,T^*];L^p(\Omega))$ and \eqref{eq:Burg} is satisfied for all $0<t<T^*$. 
	\end{Theorem}
	\begin{proof}
We use an iterative technique to prove the existence of a local mild solution of the problem \eqref{eq:Burg} (\cite[Ch. 3]{TemamNSE01}).  Let us start with the first iteration $u_0(t)=e^{-\nu t A} u_0,$ and denote
\begin{align} \label{eq:u_n+1}
u_{n+1}(t)=u_0(t)+G(u_n)(t), \, n=0,1,2,\ldots,
\end{align}
where
\begin{align}
G(u_n)(t)= - \int_0^t e^{-\nu (t-s)A} B(u_n(s))ds + \int_0^t e^{-\nu (t-s)A} f(s) ds.
\end{align}
Note that $$\|u_0(t)\|_{L^p(\Omega)} = \|e^{-\nu t A}u_0\|_{L^p(\Omega)} \le C \|u_0\|_{L^p(\Omega)}<\infty.$$ Furthermore, by using \eqref{3p3}, we have 
\begin{align}\label{3p7}
\|u_{n+1}(t)\|_{L^p(\Omega)} & \le \|u_0(t)\|_{L^p(\Omega)} + \int_0^t \|e^{-\nu (t-s)A }B(u_n(s))\|_{L^p(\Omega)} ds \nonumber \\
& \qquad \quad +\int_0^t \|e^{-\nu (t-s)A }f(s)\|_{L^p(\Omega)} ds\nonumber\\
& \le C \Big(\|u_0\|_{L^p(\Omega)} +  \int_0^t (t-s)^{-\frac{1}{2}-\frac{d}{2p}} \|u_n(s)\|_{L^p(\Omega)}^2 ds \nonumber \\
& \qquad \quad  +  \int_0^t (t-s)^{ -\frac{d}{2} (\frac{1}{2} - \frac{1}{p}) }\|f(s)\|ds\Big) \nonumber \\
& \le C\|u_0\|_{L^p(\Omega)} + C \sup_{s\in [0,t]}\|u_n(s)\|_{L^p(\Omega)}^2 \int_0^t (t-s)^{-\frac{1}{2}-\frac{d}{2p}} ds \nonumber \\
& \qquad + C \sup_{s\in [0,t]}\|f(s)\| \int_0^t (t-s)^{ -\frac{d}{2} (\frac{1}{2} - \frac{1}{p}) } ds.
\end{align}
It is easy to check that $\int_0^t (t-s)^{-\frac{1}{2}-\frac{d}{2p}} ds$ is finite when $p>d$ and $\int_0^t (t-s)^{ -\frac{d}{2} (\frac{1}{2} - \frac{1}{p}) } ds$ is finite provided $\frac{1}{p} > \frac{d-4}{2d}.$ Therefore
\begin{align}
&\sup_{t\in [0,T]} \|u_{n+1}(t)\|_{L^p(\Omega)}\nonumber\\& \le C\|u_0\|_{L^p(\Omega)} + C \|f\|_{L^\infty(0,T;L^2(\Omega))} T^{1 - \frac{d}{2} (\frac{1}{2}-\frac{1}{p})} + C \sup_{t\in [0,T]} \|u_n(t)\|_{L^p(\Omega)}^2 T^{\frac{1}{2} - \frac{d}{2p}}.
\end{align}

\noindent For all $n=1,2,\ldots,$ set $$\displaystyle f_n:=\sup_{t\in [0,T]} \|u_n(t)\|_{L^p(\Omega)} \text{ and } f_0:=\|u_0\|_{L^p(\Omega)} + C \|f\|_{L^\infty(0,T;L^2(\Omega))} T^{1 - \frac{d}{2} (\frac{1}{2}-\frac{1}{p})}$$ to obtain 
\begin{align}
f_{n+1} \le f_0 + C T^{\frac{1}{2} - \frac{d}{2p}} f_n^2 \ \text{ for all }\ n=1,2,\ldots .
\end{align}
Now, we fix $T_1^* <T $ such that 
\begin{align*}
    f_0 \le \frac{1}{8C T^{\frac{1}{2}-\frac{d}{2p}}}, \text{ or in particular, we take }T_1^*=\frac{1}{(8Cf_0)^{2p/(p-d)}}.
\end{align*}
Using induction, one can easily show that $f_n \le \frac{1}{4C T^{\frac{1}{2}-\frac{d}{2p}}}$ for all $n\ge 1,$ that is, the sequence $\{f_n\}_{n\in \mathbb{N}}$
 is uniformly bounded. 

We define $v_{n+2}(\cdot),$ $n=0,1,2,\ldots,$ by using   \eqref{eq:u_n+1} as
\begin{align*}
    v_{n+2}(t) :=u_{n+2}(t)-u_{n+1}(t) & = - \int_0^t e^{-\nu (t-s)A} B(u_{n+1}(s))ds + \int_0^t e^{-\nu (t-s)A} B(u_n(s))ds \\
    & = - \int_0^t e^{-\nu (t-s)A} (B(u_{n+1}(s)) - B( u_n(s)))ds.
\end{align*}
Once again \eqref{3p2p5}-\eqref{3p3} yield
\begin{align}\label{eqn:est vn+2}
  &  \|v_{n+2}(t)\|_{L^p(\Omega)}\nonumber \\& \le \int_0^t \|e^{-\nu (t-s)A} (B(u_{n+1}(s)) - B(u_n(s)))\|_{L^p(\Omega)} ds\nonumber \\
    & \le \int_0^t \|e^{-\nu (t-s)A} u_{n+1}(s)\xi \cdot \nabla \left( u_{n+1}(s) - u_n(s)\right) \nonumber\\
    & \qquad \qquad + e^{-\nu (t-s)A} \xi \left( u_{n+1}(s) - u_n(s)\right) \cdot \nabla u_n(s)\|_{L^p(\Omega)} ds\nonumber \\
    & \le C \int_0^t (t-s)^{-\frac{1}{2}-\frac{d}{2p}}  \left(\|u_{n+1}(s)\|_{L^p(\Omega)} + \|u_n(s))\|_{L^p(\Omega)}\right) \|u_{n+1}(s) - u_n(s))\|_{L^p(\Omega)} ds\nonumber \\
    & \le C\left( \sup_{s\in [0,t]} \|u_{n+1}(s)\|_{L^p(\Omega)} + \sup_{s\in [0,t]} \|u_{n}(s)\|_{L^p(\Omega)} \right) \sup_{s\in [0,t] } \|u_{n+1}(s) - u_n(s))\|_{L^p(\Omega)} \nonumber\\
    & \qquad \qquad \times \int_0^t (t-s)^{-\frac{1}{2}-\frac{d}{2p}} ds \nonumber\\
    & \le K \sup_{s\in [0,t]} \|v_{n+1}(s)\|_{L^p(\Omega)} T^{\frac{1}{2}-\frac{d}{2p}} \le \cdots \le  K^{n+1} (T^{\frac{1}{2}-\frac{d}{2p}})^{n+1}, \, n=0,1,2,\ldots,
\end{align} 
where $K = \displaystyle C\left( \sup_{s\in [0,t]} \|u_{n+1}(s)\|_{L^p(\Omega)} + \sup_{s\in [0,t]} \|u_{n}(s)\|_{L^p(\Omega)} \right).$
Now, note that 
\begin{align} \label{eqn:sum u_n}
	u_n(t)=u_0(t)+\sum_{m=0}^{n-1} v_{m+1}(t)
\end{align}
 and therefore
 \begin{equation} \label{3p11}
\begin{aligned}
\sup_{t\in [0,T]} \|u_n(t)\|_{L^p(\Omega)} & \le \sup_{t\in [0,T]} \left( \|u_0(t)\|_{L^p(\Omega)} + \sum_{m=1}^{n-1}\|v_{m+1}(t)\|_p\right) \\
& \le \|u_0\|_{L^p(\Omega)} + \sum_{m=0}^{n-1} K^{m+1} (T^{\frac{1}{2}-\frac{d}{2p}})^{m+1}.
\end{aligned}
\end{equation}
Note that as $n\to\infty$, the series in the  right-hand side of \eqref{eqn:sum u_n} converges, provided  $KT^{\frac{1}{2}-\frac{d}{2p}}<1,$ that is, $T< \frac{1}{K^{\frac{2p}{p-d}}}=:T_2^*.$  Therefore, for all $0<t<T^*,$ where $T^*:=\min\{T_1^*, T_2^*\},$ we denote the sum as $u(t)$ as 
\begin{align}\label{3p13}
	u(t):=\lim_{n\rightarrow \infty}u_n(t) \text{ in }L^p(\Omega)\text{ for all }t\in (0,T^*). 
\end{align}  
By using the estimate \eqref{3p3}, a calculation similar to \eqref{eqn:est vn+2} and the uniform convergence of $u_n(t)$ to $u(t)$ , one can show that $G(u_n(t))\to G(u(t))$ as $n\to\infty$.  Therefore, we infer 
\begin{align*}
	u(t)=u_0+G(u(t)), \text{ in }L^p(\Omega) \text{ for all } t\in (0,T^*),
\end{align*}
which is a mild solution of the problem \eqref{eq:Burg}. 

\medskip 
\noindent \textbf{Uniqueness.}  Let $u_1(\cdot)$ and $u_2(\cdot)$ be two mild solutions of \eqref{eq:Burg} for $0<t<T^*.$ Then $u_1(t)$ and $u_2(t)$ satisfy in $L^p(\Omega)$
\begin{align*}
	u_i(t) = e^{-\nu t A}u_0 - \int_0^t e^{-\nu (t-s)A} B (u_i(s))ds + \int_0^t e^{-\nu (t-s)A} f(s) ds\  \text{ for all } \ 0<t<T^*, 
\end{align*}
for $ i=1,2.$ Subtracting one among another and proceeding in an analogous way as in \eqref{eqn:est vn+2}, we obtain
\begin{align*}
	\|u_1(t)-u_2(t)\|_{L^p(\Omega)} &  \le C \left( \sup_{s\in [0,T^*]}\|u_1(s)\|_{L^p(\Omega)} +  \sup_{s\in [0,T^*]}\|u_2(s)\|_{L^p(\Omega)}\right) \\
	& \qquad \quad \times \sup_{s\in [0,T^*]} \|u_1(s) - u_2(s)\|_{L^p(\Omega)} (T^*)^{\frac{1}{2}-\frac{d}{2p}}  \\
	& \le K_1 (T^*)^{\frac{1}{2}-\frac{d}{2p}} \sup_{s\in [0,T^*]} \|u_1(s) - u_2(s)\|_{L^p(\Omega)} ,
\end{align*}
where $ \displaystyle K_1= C \left( \sup_{s\in [0,T^*]}\|u_1(s)\|_{L^p(\Omega)} +  \sup_{s\in [0,T^*]}\|u_2(s)\|_{L^p(\Omega)}\right).$ Now, one can choose $T^*$ such that $ K_1 (T^*)^{\frac{1}{2}-\frac{d}{2p}}<1.$ Thus $u\in C([0,T^*]; L^p(\Omega))$ is the unique mild solution of \eqref{eq:Burg}.
\end{proof}

\subsection{Global solution} 
Let us now prove that the mild solution of the problem \eqref{eq:Burg} is regular  and it  exists globally. 
\begin{Definition}[Strong solution]
	A function $u\in C([0,T]; H^1_0(\Omega))\cap L^2(0,T; H^2(\Omega))$ with $\partial_tu\in L^2(0,T; L^2(\Omega))$ is called a \emph{strong solution} to the system \eqref{eq:Burg} if for $u_0\in H^1_0(\Omega), \, f\in L^2(0,T; L^2(\Omega)),$ $u(\cdot)$ satisfies \eqref{eq:Burg} in $L^2(\Omega)$ for a.e. $t\in [0,T].$
	\end{Definition}
	
	\begin{Remark}\label{Rem3.5}
		The term $\int_0^t (t-s)^{ -\frac{d}{2} (\frac{1}{2} - \frac{1}{p}) }\|f(s)\|ds $ appearing in \eqref{3p7} can be estimated in the following way also:
		\begin{align*}
			\int_0^t (t-s)^{ -\frac{d}{2} (\frac{1}{2} - \frac{1}{p}) }\|f(s)\|ds &\leq \left(\int_0^t (t-s)^{ -d (\frac{1}{2} - \frac{1}{p}) }ds\right)^{1/2}\left(\int_0^t\|f(s)\|^2 ds\right)^{1/2}\nonumber\\&\leq CT^{1-d (\frac{1}{2} - \frac{1}{p})}\left(\int_0^T\|f(s)\|^2ds\right)^{1/2}<\infty,
		\end{align*}
		provided $2\leq p<\frac{2d}{d-2}$. That is, Theorem \ref{mild-solution} holds true for $f\in L^2(0,T;L^2(\Omega))$ also. 
	\end{Remark}
	
\begin{Theorem} \label{thm:GblSolHomBurg}
	For any $u_0 \in H^1_0(\Omega)$ and  $f\in L^2(0,T;L^2(\Omega)),$ there exists a unique strong solution 
	$$u\in L^{\infty}(0,T;H_0^1(\Omega))\cap L^2(0,T;D(A))$$ 
to 	the problem \eqref{eq:Burg} satisfying
\begin{align}\label{eqn-ener}
	\sup_{t\in[0,T]}\|\nabla u(t)\|^2+\nu\int_0^T\|Au(t)\|^2dt\leq C(\|\nabla u_0\|,\|f\|_{L^2(0,T;L^2(\Omega))},\nu,T).
\end{align} 
	\end{Theorem}
	\begin{proof}
Let us take $u_0 \in H^1_0(\Omega).$ Since $C_c^\infty(\Omega)$ is dense in $H^1_0(\Omega),$ we can find a sequence $u_0^n \in C_c^\infty(\Omega)$ such that $u_0^n \longrightarrow u_0$ in $H^1_0(\Omega)$ as $n \rightarrow \infty.$ Once again by a density argument, for any $f\in L^2((0,T)\times \Omega),$ we can find a sequence $f^n\in C_c^\infty((0,T)\times\Omega))$ such that $f^n\longrightarrow f$ in $L^2((0,T)\times \Omega)\cong L^2(0,T;L^2(\Omega))$.  Let us consider the following problem: 
\begin{equation}\label{3.11}
	\left\{
\begin{aligned}
& \frac{du^n(t)}{dt}+\nu Au^n(t)+B(u^n(t)) = f^n(t) \ \text{ for  } \ t\in (0,T),\\
& u^n(0)= u_0^n.
\end{aligned}
\right. 
\end{equation}
By Theorem \ref{mild-solution} and Remark \ref{Rem3.5},  for any $d< p<\frac{2d}{d-2}$, we obtain the existence of a local mild solution to the problem \eqref{3.11} such that 
	\begin{align}
		\sup_{t\in[0,T^*]}\|u^n(t)\|_{L^p(\Omega)}\leq C(1+\|u_0^n\|_{L^p(\Omega)}+\|f^n\|_{L^{2}((0,T)\times\Omega)}).
	\end{align}
	By Sobolev's embedding, we know that $d< p<\frac{2d}{d-2}$, the space $H_0^1(\Omega)\hookrightarrow L^p(\Omega)$ and the fact that $u_0^n \longrightarrow u_0$ in $H^1_0(\Omega)$ and $f_n\longrightarrow f$ in $L^2((0,T)\times \Omega)$ as $n \rightarrow \infty,$ we deuce from the above relation that 
		\begin{align}
		\sup_{t\in[0,T^*]}\|u^n(t)\|_{L^p(\Omega)}\leq C(1+\|u_0\|_{L^p(\Omega)} + \|f\|_{L^2(0,T;L^2(\Omega))})
	\end{align}
	and a calculation similar to \eqref{3p13} yields 
	\begin{align}\label{3p17}
	u^n \longrightarrow u\text{ in } C([0, T^*]; L^p(\Omega)).
	\end{align}
	Moreover, $u$ is the local mild solution of the problem \eqref{eq:Burg} satisfying \eqref{1p4}.   
	
	Since $u_0^n \in C_c^\infty(\Omega)$  and $f^n\in C_c^\infty((0,T)\times\Omega)),$  by using a standard parabolic theory (cf. \cite[Theorem 7.4, Chapter V]{Lady68}), we can find a smooth (or classical) solution for the  problem \eqref{3.11}. 

Taking the inner-product with $|u^n|^{p-2}u^n,$ we find
\begin{align}\label{eq:wkform_p-2}
& \left( \frac{du^n}{dt}, |u^n|^{p-2}u^n\right)+\nu \left( Au^n, |u^n|^{p-2}u^n\right) + \left(B(u^n), |u^n|^{p-2}u^n\right) = \left(f^n , |u^n|^{p-2}u^n\right).
\end{align}
Performing  integration by parts, we have 
\begin{align*}
    \left( Au^n, |u^n|^{p-2}u^n \right) & = -\int_\Omega \Delta u^n |u^n|^{p-2}u^n dx = -\sum_{i=1}^d \int_\Omega \frac{\partial^2 u^n}{\partial x_i^2} |u^n|^{p-2}u^n dx \\
    & =\sum_{i=1}^d\int_\Omega \frac{\partial u^n}{\partial x_i} \left( |u^n|^{p-2} \frac{\partial u^n}{\partial x_i} + \frac{p-2}{2}u^n  (|u^n|^2)^{p-2/2} 2 u^n \frac{\partial u^n}{\partial x_i}\right) dx \\
    & =(p-1) \sum_{i=1}^d\int_\Omega |u^n|^{p-2} \bigg(\frac{\partial u^n}{\partial x_i}\bigg)^2 dx.
\end{align*}
Furthermore, using  integration by parts again 
\begin{align*}
    \left( B(u^n), |u^n|^{p-2}u^n\right)& = \left( \frac{1}{2}\nabla \cdot (\xi (u^n)^2), |u^n|^{p-2}u^n\right)  = \frac{1}{2}\left( \sum_{i=1}^d \frac{\partial }{\partial x_i}(u^n)^2,  |u^n|^{p-2}u^n \right) \\
   & = -\frac{p-1}{2} \sum_{i=1}^d \left( (u^n)^2, |u^n|^{p-2} \frac{\partial u^n}{\partial x_i}\right) = -\frac{p-1}{2} \left( B(u^n), |u^n|^{p-2}u^n\right) ,
\end{align*}
so that $ \left( B(u^n), |u^n|^{p-2}u^n\right)=0.$ Let us now estimate right hand side of \eqref{eq:wkform_p-2}. Note that for $2< p<\frac{2d}{d-2}$, by using Sobolev's embedding, H\"older's and Young's inequalities, we deduce 
\begin{align*}
\big\vert\left(f^n , |u^n|^{p-2}u^n\right)\big\vert & = \Big\vert\left(|u^n|^{\frac{p-2}{2}}f^n , |u^n|^{\frac{p-2}{2}}u^n\right)\Big\vert \\
& \le \big\| \nabla(|u^n|^{\frac{p-2}{2}}u^n)\big\| \big\| |u^n|^{\frac{p-2}{2}}f^n\big\|_{H^{-1}(\Omega)} \\
& \le \frac{Cp}{2}\big\||u^n|^{\frac{p-2}{2}}\nabla u^n\big\| \big\| |u^n|^{\frac{p-2}{2}} f^n \big\|_{L^{\frac{p}{p-1}}(\Omega)} \\
& \le C \big\||u^n|^{\frac{p-2}{2}}\nabla u^n \big\| \big\| |u^n|^{\frac{p-2}{2}} \big\|_{L^{\frac{2p}{p-2}}(\Omega)}\|f^n\| \\
& \le C \big\||u^n|^{\frac{p-2}{2}}\nabla u^n \big\| \big\|u^n\big\|_{L^p(\Omega)}^{\frac{p-2}{2}}\|f^n\| \\
& \le \frac{(p-1)\nu}{2} \big\||u^n|^{\frac{p-2}{2}}\nabla u^n \big\|^2 + C \big\|u^n \big\|_{L^p(\Omega)}^{p-2}\|f^n\|^2.
\end{align*}
The case of $p=2$ is easy. Combining these and using \eqref{eq:wkform_p-2}, we obtain for a.e. $t\in(0,T)$
\begin{align*}
\frac{1}{p}\frac{d}{dt}\|u^n(t)\|_{L^p(\Omega)}^p+ \frac{(p-1)\nu}{2} \big\||u^n|^{\frac{p-2}{2}}\nabla u^n\big\|^2 \le C\|u^n\|_{L^p(\Omega)}^{p-2}\|f^n\|^2.
\end{align*}
For $y(t)=\|u^n(t)\|_{L^p(\Omega)}^p$, the above relation can be expressed as 
\begin{align}
	y(t)\leq y(0)+C\int_0^ty(s)^{\frac{p-2}{p}}\|f^n(s)\|^2ds,
\end{align}
for all $t\in[0,T]$. Since $0\leq \frac{p-2}{p}<1$, by using a nonlinear generalization of Gronwall's inequality (Theorem \ref{lem:NonlinGronwall}), we deduce  for all $0\leq t\leq T$, 
\begin{align}\label{3p20}
	\|u^n(t)\|_{L^p(\Omega)}^p\leq \left\{\|u_0^n\|_{L^p(\Omega)}^2+C\int_0^t\|f^n(s)\|^2ds\right\}^{\frac{p}{2}}\leq C\left\{\|u_0\|_{L^p(\Omega)}^p+\left(\int_0^T\|f(t)\|^2dt\right)^{\frac{p}{2}}\right\}. 
\end{align}
Furthermore, we have 
\begin{align}\label{3p21}
\int_0^T	\big\||u^n(s)|^{\frac{p-2}{2}}\nabla u^n(s)\big\|^2ds\leq C\left\{\|u_0\|_{L^p(\Omega)}^p+\left(\int_0^T\|f(t)\|^2dt\right)^{\frac{p}{2}}\right\}. 
\end{align}
Therefore, for $2\leq p<\frac{2d}{d-2}$,  the solution exists for all $t\in [0,T],$ and $u \in C([0,T];L^p(\Omega))$. 
Taking the inner product with $Au^n,$ we find 
\begin{align*}
   & \left( \frac{du^n}{dt}, A u^n \right) + \nu \left( A u^n, A u^n  \right)+\left(B(u^n), A u^n \right)=\left(f^n, A u^n \right),   \\
  \implies &  \frac{1}{2} \frac{d}{dt}\|\nabla u^n\|^2+\nu \|Au^n\|^2 = -(B(u^n),Au^n)+ \left(f^n, A u^n \right).
\end{align*}
Note that $\left(f^n, A u^n \right) \le \|f^n\|\|Au^n\| \le \frac{\nu}{4}\|Au^n\|^2+\frac{1}{\nu}\|f^n\|^2$ and 
$|-(B(u^n),Au^n)| \le \|B(u^n)\| \|Au^n\|.$ 
Now, for $d=2,$ using Agmon's inequality, we infer 
\begin{align*}
    \|B(u^n)\| = \|\xi u^n\cdot \nabla u^n\| \le \|u^n\|_{\infty} \|\nabla u^n\| \le C \|u^n\|^{1/2}\|Au^n\|^{1/2}\|\nabla u^n\|,
\end{align*}
and therefore
\begin{align*}
    |-(B(u^n),Au^n)| \le  \frac{\nu}{4}\|Au^n\|^2 + \frac{C}{\nu} \|u^n\|^2\|\nabla u^n\|^4.
\end{align*}
Combining all these, we obtain
\begin{align*}
    \frac{d}{dt}\|\nabla u^n\|^2 +\nu \|Au^n\|^2 \le \frac{C}{\nu}\|u^n\|^2\|\nabla u^n\|^4+\frac{C}{\nu}\|f^n\|^2,
\end{align*}
and on integrating, we deduce
\begin{align*}
    \|\nabla u^n(t)\|^2 + \nu \int_0^t \|Au^n(s)\|^2 ds \le\|\nabla u_0^n\|^2+\frac{C}{\nu}\int_0^t\|f^n(s)\|^2ds+\frac{C}{\nu}\int_0^t \|u^n(s)\|^2\|\nabla u^n(s)\|^4ds.
\end{align*}
An application of Grownwall's inequality yields 
\begin{align}
	& \|\nabla u^n(t)\|^2 + \nu \int_0^t \|Au^n(s)\|^2 ds \nonumber\\&\leq \left(\|\nabla u_0^n\|^2+\frac{C}{\nu}\int_0^T\|f^n(s)\|^2ds\right)\exp\left(\frac{C}{\nu}\int_0^t \|u^n(s)\|^2\|\nabla u^n(s)\|^2ds\right),
\end{align}
and the right hand side of above inequality is finite and is independent of $n$ by using  \eqref{3p20} and \eqref{3p21}.

Now, we consider the case $d=3$. By using Theorem \ref{thm:G-NinBdd}, we find 
	\begin{align*}
		\|\nabla u\|_{L^{\frac{2p}{p-2}}(\Omega)}\leq C\|u\|_{L^p(\Omega)}^{\frac{2(p-3)}{p+6}}\|Au\|^{\frac{12-p}{p+6}},
	\end{align*}
provided $3\leq p\leq 12$. 	Therefore, we calculate
	\begin{align*}
     \|B(u^n)\| = \|\xi u^n\cdot \nabla u^n\| \le \|u^n\|_{L^p(\Omega)} \|\nabla u^n\|_{L^{\frac{2p}{p-2}}(\Omega)} \le C_a \|u^n\|_{L^p(\Omega)}^{\frac{3p}{p+6}}\|Au^n\|^{\frac{12-p}{p+6}}.
	\end{align*}
	By using the Cauchy-Schwarz and Young's inequalities with exponent $\frac{p+6}{9}$ and $\frac{p+6}{p-3}$, we estimate 
	\begin{align*}
		|-(B(u^n),Au^n)| \le \|u^n\|_{L^p(\Omega)}^{\frac{3p}{p+6}}\|Au^n\|^{\frac{18}{p+6}} \le\frac{\nu}{4}\|Au^n\|^2 + \frac{C}{\nu} \|u^n\|_{L^p(\Omega)}^{\frac{p+6}{p-3}}
		\| u^n\|_{L^p(\Omega)}^2.
	\end{align*}
Combining all these and using the Sobolev embedding, we obtain for $3<p\leq 6$
	\begin{align*}
		\frac{d}{dt}\|\nabla u^n\|^2 +\nu \|Au^n\|^2 \le \frac{C}{\nu}\|u^n\|_{L^p(\Omega)}^{\frac{p+6}{p-3}}\|\nabla u^n\|^2+ \frac{C}{\nu}\|f^n\|^2.
	\end{align*}
	By using Gronwall's inequality, we deduce 
	\begin{align}\label{3p23}
		\|\nabla u^n(t)\|^2 &+  \nu \int_0^t \|Au^n(s)\|^2 ds\nonumber \\
		& \le \left( \|\nabla u_0^n\|^2+\frac{C}{\nu} \int_0^T\|f^n(t)\|^2 d t\right) \exp\left(\frac{C}{\nu} \int_0^T\|u^n(t)\|_{L^p(\Omega)}^{\frac{p+6}{p-3}} dt\right),
\end{align}
for all $t\in [0,T]$.  The right hand side of \eqref{3p23} is finite and is independent of $n$ by using the estimate \eqref{3p20}. Thus for $u_0\in H^1_0(\Omega)$ and $f \in L^2(0,T;L^2(\Omega)),$ we have $u \in L^\infty(0,T; H^1_0(\Omega)) \cap L^2(0,T; H^2(\Omega)).$
Using the Banach-Alaoglu theorem, one can extract a subsequence (still denoted by the same symbol) of $u^n$ that $u^n \xrightarrow{w^*} \tilde{u}$ in $L^\infty(0,T;H^1_0(\Omega))$ and $u^n \xrightarrow{w} \tilde{u} $ in $L^2(0,T; D(A))$. Moreover, one can show that $\frac{du^n}{dt}$ is bounded uniformly in $n$ and $\frac{du^n}{dt}\xrightarrow{w} \frac{d\tilde{u}}{dt}$ in $L^2(0,T;L^2(\Omega))$. Therefore, an application of the Aubin-Lions Lemma gives $u^n\to \tilde{u}$ in $L^2(0,T;H_0^1(\Omega))$ and $u^n\to \tilde{u}$ in $C([0,T];L^2(\Omega))$. By using \eqref{3p17} and  the uniqueness of weak limits, one gets $u=\tilde{u}$.    With the above convergences, we can pass the limit in \eqref{3.11} and obtain that $u$ is the unique strong solution of the problem \eqref{eq:Burg}.  
\end{proof}

\subsection{Solvability of Burgers equation with non-homogenous boundary data}
In this section, we discuss the well-posedness of viscous Burgers equation with non-homogeneous boundary data. Consider the following viscous Burgers equation with non-homogeneous boundary data:
\begin{equation} \label{eqn:Burg+NHB}
	\left\{
	\begin{aligned}
	&	 \frac{\partial y}{\partial t}(x,t) - \nu \Delta y(x,t) + \sum_{i=1}^d y(x,t) \frac{\partial y}{\partial x_i}(x,t) =f(x,t) \text{ in }\Omega \times (0,T), \\
	&	 y(x,t) = g(x,t) \text{ on } \Gamma\times (0,T),\\
	&	 y(x,0) = y_0(x) \text{ in }\Omega.
	\end{aligned} \right.
\end{equation}
We start with the existence result on a heat equation with non-homogeneous boundary data:
\begin{Lemma} \label{lem:Aux-Heat+NHM}
Let $T>0$ be fixed.	Consider the heat equation 
	\begin{equation} \label{eqn:Heat+NHB}
		\left\{
		\begin{aligned}
			& \frac{\partial z}{\partial t}(x,t) - \nu \Delta z(x,t) =0 \text{ in }\Omega \times (0,T), \\
			& z(x,t) =g(x,t) \text{ on } \Gamma\times (0,T), \\
			& z(x,0) =0 \text{ in }\Omega.
		\end{aligned}\right.
	\end{equation}
	For $g\in H^{\frac{1}{2}, \frac{1}{4}}(\Gamma\times (0,T)):= L^2(0,T; H^{1/2}(\Gamma))\cap H^{1/4}(0,T; L^2(\Gamma)),$ there exists a unique weak solution of the problem \eqref{eqn:Heat+NHB}
	\[z \in L^\infty(0,T;L^2(\Omega))\cap L^2(0,T; H^1(\Omega))\ \mbox{ such that }\ \partial_tz \in L^2(0,T; H^{-1}(\Omega)).\] And, if $g\in H^{\frac{3}{2}, \frac{3}{4}}(\Gamma\times (0,T)):=L^2(0,T; H^{3/2}(\Gamma))\cap H^{3/4}(0,T; L^2(\Gamma)),$ then \eqref{eqn:Heat+NHB} admits a unique strong solution \[z\in L^\infty(0,T; H^1(\Omega))\cap L^2(0,T; H^2(\Omega))\ \text{ such that }\ \partial_tz \in L^2(0,T; L^2(\Omega)).\]

	Furthermore, the following estimates hold:
	\begin{align} \label{eq:energyestNHH}
		&\sup_{[0,T]}\|z(t)\|^2 + \nu \int_0^T\|z(s)\|^2_{H^1(\Omega)} \, ds \le C\|g\|_{H^{\frac{1}{2}, \frac{1}{4}}(\Gamma\times (0,T))}^2,\\
		&\sup_{t\in[0,T]}\|z(t)\|_{H^1(\Omega)}^2+\nu\int_0^T\|z(t)\|_{H^2(\Omega)}^2dt\leq C\|g\|_{H^{\frac{3}{2}, \frac{3}{4}}(\Gamma\times (0,T))}^2,  \label{eq:energyestNHH1}
	\end{align}
	for some $C= C(\nu,Tr)>0.$
\end{Lemma}

\begin{proof}
	The existence of a weak solution follows from \cite[Section 15.5]{LioMag-V2}. Indeed, for a given $g\in H^{\frac{1}{2}, \frac{1}{4}}(\Gamma\times (0,T)),$ we take $m=1$ and $j=0$ in \cite[(15.33) - (15.38)]{LioMag-V2} to obtain $z\in L^2(0,T; H^1(\Omega))$ with $\partial_tz\in L^2(0,T; H^{-1}(\Omega))$ satisfying 
	\begin{align*}
		\|z\|_{L^2(0,T;H^1(\Omega))}+\|\partial_tz\|_{L^2(0,T;H^{-1}(\Omega))} \le C \|g\|_{ H^{\frac{1}{2}, \frac{1}{4}}(\Gamma\times (0,T))},
	\end{align*}
	for some $C>0.$ Now, to show $z\in L^\infty(0,T; L^2(\Omega))$ and to obtain the energy estimate \eqref{eq:energyestNHH}, we use \cite[Lemma 1.2, Page 176]{TemamNSE01}. Note that, we have $H^1_0(\Omega) \subset H^1(\Omega) \subset L^2(\Omega)\subset H^{-1}(\Omega)$ and $z \in L^2(0,T; H^1(\Omega))$ with $\frac{\partial z}{\partial t} \in L^2(0,T; H^{-1}(\Omega)).$ Therefore, from \cite[Lemma 1.2, Page 176]{TemamNSE01}, $z$ is a.e. equal to a function in $C([0,T];L^2(\Omega))$ and 
	\begin{align*}
		\frac{d}{dt}\|z(t)\|^2 = 2 \left\langle \frac{\partial z (t)}{\partial t} , z(t)\right\rangle 
	\end{align*}
	for a.e. $t\in[0,T]$. From \eqref{eqn:Heat+NHB} and using integration by parts, we obtain 
	\begin{align}\label{abscty1}
		\frac{d}{dt}\|z(t)\|^2 = -2\nu\|\nabla z(t)\|^2+\nu\langle\nabla z(t)\cdot\textbf{n},g(t)\rangle_{H^{-1/2}(\Gamma), H^{1/2}(\Gamma)},
	\end{align}
	for a.e. $t\in[0,T]$. By using the definition of trace operator, Cauchy-Schwarz and Young's inequalities, we calculate
	\begin{align}\label{abscty2}
	|\langle\nabla z\cdot\textbf{n},g\rangle_{H^{-1/2}(\Gamma), H^{1/2}(\Gamma)}|&\le \|\nabla z\|_{H^{-1/2}(\Gamma)}\|g\|_{H^{1/2}(\Gamma)} \nonumber\\& \le \|z\|_{H^{1/2}(\Gamma)} \|g\|_{H^{1/2}(\Gamma)} \nonumber\\
	&  \le C_{Tr}\|z\|_{H^1(\Omega)} \|g\|_{H^{1/2}(\Gamma)}\nonumber\\& \le \|z\|_{H^1(\Omega)}^2+\frac{C_{Tr}}{4}\|g\|_{H^{\frac{1}{2}}(\Gamma)}^2.
	\end{align}
Combining \eqref{abscty1}-\eqref{abscty2}, we obtain 
\begin{align*}
		\frac{d}{dt}\|z(t)\|^2+\nu\|\nabla z(t)\|^2 \leq \nu\|z(t)\|^2+\nu\frac{C_{Tr}}{4}\|g(t)\|_{H^{1/2}(\Gamma)}^2,
\end{align*}
for a.e. $t\in[0,T]$. Finally, by using Gronwalls' inequality, we obtain the required estimate. 
	
%
	Now, for $g\in H^{\frac{3}{2}, \frac{3}{4}}(\Gamma\times (0,T)),$ we use \cite[(15.39)]{LioMag-V2} with $m=1, j=0$ to show the existence of solution $z\in H^{2,1}(\Omega\times (0,T))=L^2(0,T; H^2(\Omega))\cap H^1(0,T; L^2(\Omega))$ satisfying 
	\begin{align*}
		\|z\|_{L^2(0,T; H^2(\Omega))}+\Big\|\frac{\partial z}{\partial t} \Big\|_{L^2(0,T; L^2(\Omega))} \le C \|g\|_{H^{\frac{3}{2}, \frac{3}{4}}(\Gamma\times (0,T))},
	\end{align*}
	for some $C>0.$ Since, $z\in L^2(0,T; H^2(\Omega))$ with $\partial_tz \in L^2(0,T; L^2(\Omega)),$ \cite[Theorem 4, Chapter 5, Section 5.9]{Eva} yields $z\in L^\infty (0,T; H^1(\Omega))$ and 
	\begin{align*}
		\|z\|_{L^\infty(0,T; H^1(\Omega))} \le C\left( \|z\|_{L^2(0,T; H^2(\Omega))} + \Big\|\frac{\partial z}{\partial t} \Big\|_{L^2(0,T; L^2(\Omega))} \right) \le C (\nu, T, Tr, \|g\|_{H^{\frac{3}{2}, \frac{3}{4}}(\Gamma\times (0,T))}),
	\end{align*}
	which completes the proof.
\end{proof}

In the next theorem, we discuss the well-posedness of a Burgers equation with non-homogeneous boundary data. 

\begin{Theorem} \label{th:NonHomNStBurgWP}
     For fixed $T>0,$ $y_0 \in H^1_0(\Omega)$ and $g\in H^{\frac{3}{2},\frac{3}{4}}(\Gamma\times (0,T)):= L^2(0,T; H^{3/2}(\Gamma))\cap H^{3/4}(0,T; L^2(\Gamma)),$ there exists a unique strong solution (of \eqref{eqn:Burg+NHB}) $$y \in L^\infty(0,T;H^1(\Omega))\cap L^2(0,T; H^2(\Omega)) \text{ with } \frac{\partial y}{\partial t}  \in L^2(0,T;L^2(\Omega)).$$ 
	
\end{Theorem}

\begin{proof}
	Let us start by considering the following non-homogeneous heat equation
	\begin{equation}
		\left\{
		\begin{aligned}
			& \frac{\partial z}{\partial t}(x,t) - \nu \Delta z(x,t)=0 \text{ in } \Omega \times (0,T), \\
			& z(x,t)=g(x,t) \text{ on } \Gamma \times (0,T), \\
			& z(x,0)=0 \text{ in } \Omega. 
		\end{aligned} \right.
	\end{equation}
	The well-posedness of this non-homogeneous heat equation is discussed in Lemma \ref{lem:Aux-Heat+NHM}. \\
	Let us take a change of variable $w:=y-z$ and then $w$ satisfies
	
	\begin{equation} \label{eq:homoBurgLower}
		\left\{
		\begin{aligned}
			& \frac{\partial w}{\partial t} - \nu \Delta w  + \sum_{i=1}^d w \frac{\partial w}{\partial x_i} +  \sum_{i=1}^d  \frac{\partial z}{\partial x_i} w+ \sum_{i=1}^d z \frac{\partial w}{\partial x_i} =f - \sum_{i=1}^d z \frac{\partial z}{\partial x_i} \text{ in } \Omega \times (0,T), \\
			& w(x,t)=0 \text{ on } \Gamma \times (0,T), \\
			& w(x,0)=y_0(x) \text{ in } \Omega. 
		\end{aligned}\right.
	\end{equation}
	Now, the existence of a unique strong solution of \eqref{eq:homoBurgLower} can be established by following the proof of Theorem \ref{thm:GblSolHomBurg} along with the fact that $\displaystyle f, \sum_{i=1}^d z \frac{\partial z}{\partial x_i} \in L^2(0,T;L^2(\Omega)).$ Indeed,
	\begin{align*}
		\left\| \sum_{i=1}^d z \frac{\partial z}{\partial x_i}\right\|_{L^2(0,T; L^2(\Omega))}^2 & = \int_0^T \left\| \sum_{i=1}^d z \frac{\partial z}{\partial x_i}\right\|^2 dt \\
		& \le \int_0^T \|z(t)\|_{L^\infty(\Omega)}^2 \|\nabla z(t)\|^2 dt \\
		&\le \sup_{[0,T]} \|\nabla z(t)\|^2 \int_0^T\| z(t)\|^2_{L^\infty(\Omega)} dt.
	\end{align*}
	Now, using Agmons' inequality (\cite[Lemma 13.2]{Agmon10}) and Lemma \ref{lem:Aux-Heat+NHM}, we have 
	\begin{align*}
		\left\| \sum_{i=1}^d z \frac{\partial z}{\partial x_i}\right\|_{L^2(0,T; L^2(\Omega))}^2 & \le C
		\begin{cases}
			\displaystyle\sup_{[0,T]} \|\nabla z(t)\|^2 \int_0^T \|z(t)\|\|z(t)\|_{H^2(\Omega)} dt, & \text{ for } d=2, \\
			\displaystyle\sup_{[0,T]} \|\nabla z(t)\|^2 \int_0^T \|z(t)\|^{1/2}\|z(t)\|_{H^2(\Omega)}^{3/2} dt ,& \text{ for } d=3, \\
		\end{cases} \\
		& \le C\|g\|_{H^{\frac{3}{2},\frac{3}{4}}(\Gamma\times (0,T))}.
	\end{align*}
As $y_0\in H^1_0(\Omega),$ by using a similar argument as used to establish Theorem \ref{thm:GblSolHomBurg}, we obtain a unique strong solution $w\in L^\infty(0,T; H^1_0(\Omega)) \cap L^2(0,T; H^2(\Omega))$ of \eqref{eq:homoBurgLower}.  Hence, \eqref{eqn:Burg+NHB} admits a unique strong solution $y=w+z \in  L^\infty(0,T; H^1(\Omega)) \cap L^2(0,T; H^2(\Omega)).$
\end{proof}

\section{Error estimates for non-stationary viscous Burgers equation}   \label{sec:errestNStVB}
This section is devoted to discuss the error analysis for non-stationary viscous Burgers equation using deep learning method. We first show that if there exists a DNN approximate solution $u_N \in \mathcal{F}_N$ corresponding to which the loss is small, then the  the error $\|u-u_N\|$ is also small. Conversely, we also show that for any given small $\varepsilon>0,$ there exists a DNN $u_N \in \mathcal{F}_N$ for which the loss is less than $\varepsilon.$ We also discuss the stability analysis of the proposed DNN scheme.

\begin{Theorem} \label{thm:err-aux-1}
Let $u_0\in H^1_0(\Omega)$ and $ f\in L^2(0,T; L^2(\Omega)).$ Let $u$ be the strong solution of the problem \eqref{eq:Burg}, and for a given $\varepsilon>0,$ there exists $u_N\in \mathcal{F}_N$ be such that 
\begin{align}
  \bigg\|\frac{\partial u_N}{\partial t} -\nu \Delta u_N+\sum_{i=1}^d u_N \frac{\partial u_N}{\partial x_i} -f \bigg\|_{L^2(0,T;L^2(\Omega))}^2 &+ \|u_N\vert_{\Gamma}\|_{ H^{\frac12,\frac14}(\Gamma\times(0,T))}^2\nonumber\\&+ \|u_0(\cdot) - u_N(\cdot,0)\|^2 \le \varepsilon^2.
\end{align}
 Then there exists $C>0$ such that
\begin{align*}
    \sup_{t\in [0,T]} \|u(t) - u_N(t)\|^2+\nu\int_0^T\|\nabla(u(s) - u_N(s)) \|^2ds \le C\varepsilon^2.
\end{align*}
\end{Theorem}
\begin{proof}
Throughout the proof $C$ is a positive generic constant that may depend on $\nu, $ constants appearing in Poincar\'e inequality, Sobolev embedding, Agmon's and Ladyzhenskaya inequalities.  Let $\displaystyle -\wt{f}:=\frac{\partial u_N}{\partial t} -\nu \Delta u_N+\sum_{i=1}^d u_N \frac{\partial u_N}{\partial x_i} -f$ and $-\wt{g}:=u_N\vert_\Gamma.$ Then, we have 
\begin{equation}\label{eq:tilde_u_N-pde-lift}
	\left\{
\begin{aligned}
    &\frac{\partial u_N}{\partial t} -\nu \Delta u_N+\sum_{i=1}^d u_N \frac{\partial u_N}{\partial x_i} =f -\wt{f} \text{ in }\Omega\times (0,T), \\
    & u_N=-\wt{g} \text{ on }\Gamma\times (0,T).
\end{aligned}\right.
\end{equation}
Set $w=u-u_N,$ and note that 
$$ \sum_{i=1}^d u \frac{\partial u}{\partial x_i} - \sum_{i=1}^d u_N \frac{\partial u_N}{\partial x_i}= \sum_{i=1}^d w \frac{\partial u}{\partial x_i} + \sum_{i=1}^d u \frac{\partial w}{\partial x_i} - \sum_{i=1}^d w \frac{\partial w}{\partial x_i},$$
then the system satisfied by $w$ is
\begin{equation} \label{eqn:w=u-u_N}
	\left\{
\begin{aligned}
    &\frac{\partial w}{\partial t} -\nu \Delta w + \sum_{i=1}^d w \frac{\partial u}{\partial x_i} + \sum_{i=1}^d u \frac{\partial w}{\partial x_i} - \sum_{i=1}^d w \frac{\partial w}{\partial x_i} =\wt{f} \text{ in }\Omega\times (0,T), \\
    & w=\wt{g} \text{ on }\Gamma\times (0,T),\\
    & w(x,0)=u_0(x)- u_N(x,0) \text{ for all }x\in \Omega.
\end{aligned}\right.
\end{equation}
Now, for a given $\wt{g} \in  H^{\frac32,\frac34}(\Gamma)),$ we consider the following auxiliary heat equation with non-homogenous boudary data:
\begin{equation} \label{eqn:Aux-heatNHB}
	\left\{ 
	\begin{aligned}
    & \frac{\partial z}{\partial t}(x,t) - \nu \Delta z(x,t) =0 \text{ in }\Omega \times (0,T), \\
    & z(x,t)= \wt{g}(x,t) \text{ on }\Gamma\times (0,T),\\
    & z(x,0)=0 \text{ in }\Omega.
    \end{aligned} \right.
\end{equation}
The existence of a unique strong solution of the problem \eqref{eqn:Aux-heatNHB} satisfying the estimate \eqref{eq:energyestNHH1} is discussed in Lemma \ref{lem:Aux-Heat+NHM}.  

  Subtract \eqref{eqn:Aux-heatNHB} from \eqref{eqn:w=u-u_N} and set $v:= w-z$ to obtain
\begin{equation}
	\left\{ 
    \begin{aligned}
       & \frac{\partial v}{\partial t}  -\nu \Delta v +\sum_{i=1}^d (v+z) \frac{\partial u}{\partial x_i} + \sum_{i=1}^d u \frac{\partial (v+z)}{\partial x_i} - \sum_{i=1}^d (v+z) \frac{\partial (v+z)}{\partial x_i} =\wt{f} \text{ in }\Omega\times (0,T), \\
    & v=0 \text{ on }\Gamma\times (0,T),\\
    & v(x,0)=u_0(x) - u_N(x,0) \text{ for all }x\in \Omega.
    \end{aligned} \right.
\end{equation}
On simplification, one gets
\begin{equation}\label{model1}
		\left\{ 
    \begin{aligned}
        & \frac{\partial v}{\partial t} -\nu \Delta v  +\sum_{i=1}^d v \frac{\partial u}{\partial {x_i}} +\sum_{i=1}^d z \frac{\partial u}{\partial {x_i}} 
        + \sum_{i=1}^d u \frac{\partial v}{\partial {x_i}}   +\sum_{i=1}^d u \frac{\partial z}{\partial {x_i}} \\
        & \qquad - \sum_{i=1}^d v \frac{\partial v}{\partial {x_i}}-\sum_{i=1}^d v \frac{\partial z}{\partial {x_i}} - \sum_{i=1}^d z \frac{\partial v}{\partial {x_i}} - \sum_{i=1}^d z \frac{\partial z}{\partial {x_i}} =\wt{f} \text{ in } \Omega \times (0,T), \\
        & v=0 \text{ on }\Gamma \times (0,T) ,\\
        & v(x,0)=u_0(x)- u_N(x,0) \text{ for all }x\in \Omega.
    \end{aligned} \right.
\end{equation}
Taking the inner product with $v$ and performing integration by parts in the first equation of \eqref{model1} leads to, 
\begin{align}\label{eqn-ee}
    \frac{1}{2}\frac{d}{dt}\|v\|^2 +\nu \|\nabla v\|^2 & =  - \frac{1}{2} \Big(\sum_{i=1}^d v \frac{\partial u}{\partial {x_i}},v \Big) + \Big(\sum_{i=1}^d z \frac{\partial v}{\partial x_i},u\Big) + \frac{1}{2} \Big(\sum_{i=1}^d v \frac{\partial z}{\partial {x_i}},v \Big)\nonumber \\
    & \qquad       + \Big( \sum_{i=1}^d z \frac{\partial z}{\partial {x_i}},v \Big) + ( \wt{f},v)\nonumber\\
    & = : \sum_{k=1}^5 I_k.
\end{align}
We use generalized H\"{o}lder's and Ladyzhenskaya's inequalities to estimate the first term in two dimensions as 
\begin{align*}
   |I_1|&= \Big\vert \frac{1}{2} \Big(\sum_{i=1}^d v \frac{\partial u}{\partial x_i}, v\Big)  \Big\vert   \le \frac{1}{2} \|\nabla u\| \|v\|_{L^4(\Omega)}^2    \le  C\|\nabla u\| \|v\|^{2-\frac d2}\|\nabla v\|^{\frac d2}
    \\& \le  \frac{\eta\nu}{5}\|\nabla v\|^2 + C\|\nabla u\|^{\frac{4}{4-d}}\|v\|^2,
\end{align*}
where the constant $\eta>0$ will be specified later. 
 Using generalized H\"older's and Cauchy Schwarz inequalities, we estimate $I_2$ as 
\begin{align*}
    |I_2|=
    \Big\vert \Big(\sum_{i=1}^d z \frac{\partial v}{\partial x_i},u \Big) \Big\vert \le \|\nabla v\| \|z\|_{L^4(\Omega)}\|u\|_{L^4(\Omega)} \le \frac{\eta\nu}{5} \|\nabla v\|^2+ C \|z\|_{L^4(\Omega)}^2\|u\|_{L^4(\Omega)}^2.
\end{align*}



\noindent Using analogous arguments, we estimate third term as
\begin{align*}
    |I_3| = \Big\vert \frac{1}{2} \Big(\sum_{i=1}^d v \frac{\partial z}{\partial x_i},v \Big)\Big\vert \le \frac{1}{2} \|\nabla z\| \|v\|_{L^4(\Omega)}^2    \leq  \frac{\eta\nu}{5}\|\nabla v\|^2 + C\|\nabla z\|^{\frac{4}{4-d}}\|v\|^2.
\end{align*}

\noindent In a similar way, we estimate $I_4$ as 

\begin{align*}
	|I_4|= \Big\vert \Big( \sum_{i=1}^d z\frac{\partial z}{\partial {x_i}},v \Big)\Big\vert &  \le  \|z\|_{L^\infty(\Omega)} \|\nabla z\| \|v\| \leq\frac{1}{2} \|\nabla z\|^2+\frac{1}{2}\|z\|_{L^\infty(\Omega)}^2 \|v\|^2 .
\end{align*}

\noindent H\"older's inequality leads to 
$$|I_5|=|( \wt{f},v) \le C \|\wt{f}\|_{H^{-1}(\Omega)}^2 + \frac{\eta\nu}{5}\|\nabla v\|^2.$$
Now, we choose $\eta=\frac{1}{2}$ and to proceed further. We combine the above estimates  and using it in \eqref{eqn-ee} to find 
\begin{align*}
    \frac{d}{dt}\|v(t)\|^2 +\nu \|\nabla v(t)\|^2 & \le  C \left( \|\nabla u(t)\|^{\frac{4}{4-d}} + \|\nabla z(t)\|^{\frac{4}{4-d}}  + \|z(t)\|_{L^{\infty}(\Omega)}^2\right) \|v(t)\|^2  \\
    & \quad + C \left(\|\wt{f}(t)\|_{H^{-1}(\Omega)}^2  +  \|z(t)\|_{L^4(\Omega)}^2\|u(t)\|_{L^4(\Omega)}^2 + \|\nabla z(t)\|^2\right).
\end{align*}
for a.e. $t\in[0,T]$. Integrating from $0$ to $t$ and then applying Gronwall's inequality, we deduce that
\begin{align}\label{eqn-final}
   & \|v(t)\|^2 +\nu \int_0^t \|\nabla v(s)\|^2 \, ds  
   \nonumber\\&\leq \left\{\|v(0)\|^2+ C\int_0^t \left(\|\wt{f}(s)\|_{H^{-1}(\Omega)}^2  +  \|z(s)\|_{L^4(\Omega)}^2\|u(s)\|_{L^4(\Omega)}^2 + \|\nabla z(s)\|^2\right)ds\right\}\nonumber\\&\quad\times\exp\left\{C\int_0^t \left( \|\nabla u(s)\|^{\frac{4}{4-d}} + \|\nabla z(s)\|^{\frac{4}{4-d}}  + \|z(s)\|_{L^{\infty}(\Omega)}^2 \right) ds\right\},
\end{align}
for all $t\in[0,T]$. The terms appearing in the exponential are finite by using \eqref{eq:energyestNHH} and \eqref{eqn-ener}. Using hypotheses and Lemma \ref{lem:Aux-Heat+NHM}, we note the following:
\begin{align*}
	& \|v(0)\|^2 = \|u_0(x)- u_N(x,0)\|^2 \le \varepsilon^2, \\
	& \int_0^T  \|\wt{f}(s)\|_{H^{-1}(\Omega)}^2\, ds\leq C \int_0^T  \|\wt{f}(s)\|_{L^{2}(\Omega)}^2\, ds \le C\varepsilon^2, \\
	&\int_0^T\|z(s)\|_{L^4(\Omega)}^2\|u(s)\|_{L^4(\Omega)}^2 ds\leq \sup_{s\in[0,T]}\|u(s)\|_{L^4(\Omega)}^2\int_0^T\|z(s)\|_{L^4(\Omega)}^2ds\leq C\varepsilon^2,\\
	&\int_0^T \|\nabla z(s)\|^2ds\leq C\varepsilon^2. 
\end{align*}
Therefore, from \eqref{eqn-final}, we deduce that 
\begin{align*}
	&\sup_{t\in[0,T]} \|v(t)\|^2 +\nu \int_0^T \|\nabla v(t)\|^2 \, dt\leq C\varepsilon^2. 
\end{align*}
Recall that 
$
u(x,t) - u_N(x,t)= w(x,t) = v(x,t)+z(x,t)
$
and therefore
\begin{align*}
    \sup_{[0,T]}\|u(t) - u_N(t)\|^2 + \nu \int_0^T \|\nabla (u(t) - u_N(t))\|^2 dt \le C\varepsilon^2,
\end{align*}
which completes the proof. 
\end{proof}


\begin{Theorem} \label{th:existDNNerrNStVB}
Let $f \in L^2(0,T; L^2(\Omega))$ and $u_0 \in H^1_0(\Omega)$. Then for a given $\varepsilon>0,$ there exists $u_N\in \mathcal{F}_N$ such that 
\begin{align}\label{eqn-est}
    \bigg\|\frac{\partial u_N}{\partial t}-\nu \Delta u_N+ \sum_{i=1}^d u_N \frac{\partial u_N}{\partial x_i} -f \bigg\|_{L^2(0,T; L^2(\Omega))}^2  &+ \| u_N|_\Gamma\|_{L^2(0,T; H^{1/2}(\Gamma))}^2 \nonumber\\
    & + \|u(x,0) - u_N(x,0)\|^2 \le C \varepsilon^2,
\end{align}
for some $C=C(\nu, T, \|\nabla u_0\|, \|f\|_{L^2(0,T; L^2(\Omega))}, C_a, C_s)>0.$
\end{Theorem}

\begin{proof}
 Since $f \in L^2(0,T; L^2(\Omega))$ and $u_0 \in H^1_0(\Omega),$ due to Theorem \ref{thm:GblSolHomBurg}, \eqref{eq:Burg} admits a unique strong solution $u \in L^2(0,T; H^2(\Omega)\cap H^1_0(\Omega)) \cap L^\infty(0,T; H^1_0(\Omega)) \cap H^1(0,T; L^2(\Omega)).$ Since, $C_c^\infty(\Omega \times (0,T))$ is dense in $ L^2(0,T; H^2(\Omega)\cap H^1_0(\Omega)) \cap H^1(0,T; L^2(\Omega)),$ we can proceed as in the proof of Theorem  \ref{th:existDNNerrorStVB}
and find $u_N \in \mathcal{F}_N$ such that 
\begin{equation}\label{eqn:existDNN_Burge_Norm}
\begin{aligned}
   &  \|u- u_N\|_{L^2(0,T; H^2(\Omega))}+\|u- u_N\|_{H^1(0,T; L^2(\Omega))}\le \varepsilon. \\
\end{aligned}
\end{equation}
An application of \cite[Theorem 4, Chapter 5, Section 5.9]{Eva}, the second estimate in \eqref{eqn:existDNN_Burge_Norm} implies 
\begin{align}\label{eqn-con}
	\|u-u_N\|_{C([0,T]; H^1(\Omega))} \le C \big(\|u- u_N\|_{L^2(0,T; H^2(\Omega))}+\|u- u_N\|_{H^1(0,T; L^2(\Omega)} \big)\le C\varepsilon.
\end{align}
Using \eqref{eq:Burg} and triangle inequality, we have
\begin{align*}
&\bigg\| \frac{\partial u_N}{\partial t}-  \nu \Delta u_N+ \sum_{i=1}^d u_N \frac{\partial u_N}{\partial x_i} -f   \bigg\|^2 \nonumber\\& = \bigg\| \frac{\partial u_N}{\partial t}-\nu \Delta u_N+ \sum_{i=1}^d u_N \frac{\partial u_N}{\partial x_i} - \frac{\partial u}{\partial t}+\nu \Delta u - \sum_{i=1}^d u \frac{\partial u}{\partial x_i} \bigg\|^2 \\
& \le C\left( \Big\| \frac{\partial }{\partial t}  (u-u_N)\Big\|^2 + \nu \|\Delta (u-u_N)\|^2 + \bigg\| \sum_{i=1}^d u \bigg(\frac{\partial u}{\partial x_i} - \frac{\partial u_N}{\partial x_i} \bigg) \bigg\|^2 \right.\nonumber\\&\quad\left.+  \bigg\| \sum_{i=1}^d (u-u_N)\frac{\partial u_N}{\partial x_i} \bigg\|^2\right).
\end{align*}
From \eqref{eqn:existDNN_Burge_Norm} and \eqref{eqn-con}, one can easily deduce 
\begin{equation}\label{eqn:exstDNNallest}
	\begin{aligned}
		 \|u(x,0) - u_N(x,0)\| &\le \varepsilon, \\
		\Big\| \frac{\partial }{\partial t}  (u-u_N)\Big\|_{L^2(0,T; L^2(\Omega))} &\le \| u-u_N\|_{H^1(0,T; L^2(\Omega))} \le  \varepsilon , \\
		\nu \|\Delta (u-u_N)\|_{L^2(0,T; L^2(\Omega))}& \le \nu \|u-u_N\|_{L^2(0,T; H^2(\Omega))} \le  \nu\varepsilon, \\
		\bigg\| \sum_{i=1}^d  u \Big(\frac{\partial u}{\partial x_i} - \frac{\partial u_N}{\partial x_i} \Big)\bigg\|_{L^2(0,T; L^2(\Omega))}^2& \le \int_0^T \|  \nabla (u - u_N)(t)\|_{L^4(\Omega)}^2  \|u(t)\|_{L^4(\Omega)}^2 dt \\
		&  \le C \|u\|_{L^\infty(0,T; L^4(\Omega))}^2 \|  u -  u_N\|_{L^2(0,T;H^2(\Omega))}^2 \\&\le C\|u\|_{L^\infty(0,T; H^1(\Omega))} \varepsilon^2, 
	\end{aligned}
\end{equation}
%
and 
\begin{equation} 
	\begin{aligned}
		 \bigg\| \sum_{i=1}^d (u-u_N)\frac{\partial u_N}{\partial x_i} \bigg\|_{L^2(0,T; L^2(\Omega))}^2& \le \int_0^T \|(u - u_N)(t)\|_{L^\infty(\Omega)}^2  \|\nabla u_N(t)\|^2 dt \\
		&  \le C \|u-u_N\|_{L^2(0,T; H^2(\Omega))}^2 \|u_N\|_{L^\infty(0,T; H^1(\Omega))}^2 \\& \le \|u_N\|_{L^\infty(0,T; H^1(\Omega))}^2\varepsilon^2.
	\end{aligned}
\end{equation}
Furthermore, note that 
\begin{equation} \label{eq:exstDNNest23}
\begin{aligned}
    \|u_N\|_{L^\infty(0,T; H^1(\Omega))} & \le \|u-u_N\|_{L^\infty(0,T; H^1(\Omega))} + \| u\|_{L^\infty(0,T; H^1(\Omega))} \\
    &\le \varepsilon + \|u\|_{L^\infty(0,T; H^1(\Omega))} \le 2\|u\|_{L^\infty(0,T; H^1(\Omega))},
\end{aligned}
\end{equation}
and
\begin{equation}  \label{eq:exstDNNest24}
\begin{aligned}
    \|u_N|_{\Gamma}\|_{L^2(0,T; H^{1/2}(\Gamma))} &= \|u|_{\Gamma}-u_N|_{\Gamma}\|_{L^2(0,T; H^{1/2}(\Gamma))}  \\
    & =\| (u-u_N)|_{\Gamma}\|_{L^2(0,T; H^{1/2}(\Gamma))} \\
    &\le C_{Tr}\|u-u_N\|_{L^2(0,T;H^1(\Omega))} \le C_{Tr} \varepsilon.
\end{aligned}
\end{equation}
Combining the estimates in \eqref{eqn:exstDNNallest} along with the estimates \eqref{eq:exstDNNest23} and \eqref{eq:exstDNNest24}, we deduce the required result \eqref{eqn-est}. 
\end{proof}

\subsection{Stability}
In this section, we discuss the stability of the DNN scheme. We start with the following lemma.

\begin{Lemma} \label{lem:aux-4-stabScheme}
Let $u_1$ and $u_2$ be the strong solutions of 
\begin{equation*}
\left\{ 
\begin{aligned}
& \frac{\partial u_{1}}{\partial t}-\nu \Delta u_1 +\sum_{i=1}^d u_1 \frac{\partial u_1}{\partial x_i}= f_1 \text{ in } \Omega\times (0,T), \\
& u_1=0 \text{ on }\Gamma\times (0,T), \\
& u_1(x,0)=u_{1,0}(x) \text{ in }\Omega, 
\end{aligned}\right.
\end{equation*}

and
\begin{equation*}
	\left\{ 
	\begin{aligned}
& \frac{\partial u_{2}}{\partial t}-\nu \Delta u_2 +\sum_{i=1}^d u_2 \frac{\partial u_2}{\partial x_i}= f_2 \text{ in } \Omega\times (0,T), \\
& u_2=0 \text{ on }\Gamma\times (0,T), \\
& u_2(x,0)=u_{2,0}(x) \text{ in }\Omega,
\end{aligned}\right.
\end{equation*}
respectively, where $f_1,f_2 \in L^2(0,T;L^2(\Omega))$ and $u_{1,0}, u_{2,0} \in H^1_0(\Omega).$ Then the following stability  estimate holds:
\begin{align*}
   &  \sup_{t\in [0,T]}\|u_1(t)-u_2(t)\|^2+\nu \int_0^T \|\nabla (u_1(s)-u_2(s))\|^2 ds\nonumber\\& \le C(\|u_{1,0}-u_{2,0}\|^2 + \|f_1-f_2\|_{L^2(0,T; L^2(\Omega))}^2),
\end{align*}
for some constant $C=C\left(\|\nabla u_{1,0}\|,\|f_1\|_{L^2(0,T;L^2(\Omega))} \right)>0.$
\end{Lemma}

\begin{proof}
Subtracting the second equation from first and setting $u=u_1-u_2$, we find
\begin{equation}\label{model2}
	\left\{ 
	\begin{aligned}
    & \frac{\partial u}{\partial t}-\nu \Delta u + \sum_{i=1}^d u_1 \frac{\partial u_1}{\partial x_i} -  \sum_{i=1}^d u_2 \frac{\partial u_2}{\partial x_i} =f_1-f_2=:f \text{ in } \Omega \times (0,T),\\
    & u=0 \text{ on }\Gamma\times (0,T), \\
    & u(x,0)=u_{1,0}(x)-u_{2,0}(x)=:u_0(x) \text{ in }\Omega.
\end{aligned}\right.
\end{equation}
Performing integration by parts, we have 
\begin{align*}
	\bigg(\sum_{i=1}^d u_1 \frac{\partial u_1}{\partial x_i} -  \sum_{i=1}^d u_2 \frac{\partial u_2}{\partial x_i},u\bigg) &=\bigg(\sum_{i=1}^d u\frac{\partial u_1}{\partial x_i} + \sum_{i=1}^d u_2 \frac{\partial u}{\partial x_i},u\bigg)  \nonumber\\&=\bigg(\sum_{i=1}^d u\frac{\partial u_1}{\partial x_i} - \sum_{i=1}^d u\frac{\partial u}{\partial x_i}+ \sum_{i=1}^d u_1\frac{\partial u}{\partial x_i},u\bigg) \nonumber\\&= -\bigg( \sum_{i=1}^d u_1\frac{\partial u}{\partial x_i},u\bigg). 
\end{align*}
Taking the inner product with $u$ in the first equation of \eqref{model2} and applying integration by parts, it holds that
\begin{align*}
   &\frac{1}{2}\frac{d}{dt}\|u(t)\|^2+\nu \|\nabla u(t)\|^2 - \Big(\sum_{i=1}^d u_1(t) \frac{\partial u(t)}{\partial x_i},u(t)\Big)=(f(t),u(t)),
\end{align*}
for a.e. $t\in[0,T]$. Now, using generalized H\"older's and Gagliardo-Nirenberg's inequalities, we calculate 
\begin{align*}
	\Big\vert \Big(\sum_{i=1}^d u_1 \frac{\partial u}{\partial x_i},u\Big) \Big\vert & \le \|\nabla u\| \|u_1\|_{L^4(\Omega)}\|u\|_{L^4(\Omega)} \\
	& \le C \|\nabla u\|^{\frac{4+d}{4}} \|u_1\|_{L^4(\Omega)}\|u\|^{\frac{4-d}{4}} \le \frac{\nu}{4}\|\nabla u\|^2 + C \|u_1\|^{\frac{8}{4-d}}_{L^4(\Omega)}\|u\|^2.
\end{align*}
Therefore, we have for a.e. $t\in[0,T]$
\begin{align*}
    \frac{d}{dt}\|u(t)\|^2 +\nu\|\nabla u(t)\|^2 \le   C\|f(t)\|_{H^{-1}(\Omega)}^2 + C \|u_1(t)\|^{\frac{8}{4-d}}_{L^4(\Omega)}\|u(t)\|^2.
\end{align*}
Integrating from $0$ to $t$ and using Gronwall's inequality, we finally deduce 
\begin{align*}
  &  \sup_{t\in [0,T]}\|u(t)\|^2+\nu \int_0^T \|\nabla u(s)\|^2 ds \nonumber\\& \le C(\|u_0\|^2 + \|f\|_{L^2(0,T; H^{-1}(\Omega))}^2) \exp\left\lbrace\int_0^T \|u_1(t)\|^{\frac{8}{4-d}}_{L^4(\Omega)} dt\right\rbrace \\
    & \le C \left(\|u_{1,0}- u_{2,0}\|^2 + \|f_1-f_2\|_{L^2(0,T; H^{-1}(\Omega))}^2\right),
\end{align*}
for some $C=C\left(\|\nabla u_{1,0}\|,\|f_1\|_{L^2(0,T;L^2(\Omega))}\right)>0.$ This concludes the proof.
\end{proof}

\begin{Theorem} \label{th:stabNStVB}
Let $u_{N_1}\in \mathcal{F}_{N_1}$ and $u_{N_2}\in \mathcal{F}_{N_2}$ be approximated solutions of 
\begin{equation*}
	\left\{ 
	\begin{aligned}
& \frac{\partial u_{1}}{\partial t}-\nu \Delta u_1 +\sum_{i=1}^d u_1 \frac{\partial u_1}{\partial x_i}= f_1 \text{ in } \Omega\times (0,T), \\
& u_1=0 \text{ on }\Gamma\times (0,T), \\
& u_1(x,0)=u_{1,0}(x) \text{ in }\Omega, 
\end{aligned}\right.
\end{equation*}
and
\begin{equation*}
	\left\{ 
	\begin{aligned}
& \frac{\partial u_{2}}{\partial t}-\nu \Delta u_2 +\sum_{i=1}^d u_2 \frac{\partial u_2}{\partial x_i}= f_2 \text{ in } \Omega\times (0,T), \\
& u_2=0 \text{ on }\Gamma\times (0,t), \\
& u_2(x,0)=u_{2,0}(x) \text{ in }\Omega,
\end{aligned}\right.
\end{equation*}
respectively, such that the hypotheses in Theorem \ref{thm:err-aux-1} hold for any given $\varepsilon>0$. Then we have 
\begin{align*}
  &  \sup_{t\in [0,T]}\|u_{N_1}(t) - u_{N_2}(t)\|^2 +\nu \int_0^T \|\nabla (u_{N_1} - u_{N_2})(t)\|^2 dt \\&
     \qquad \le C \left( \varepsilon^2 + \|u_{1,0}-u_{2,0}\|^2 + \|f_1-f_2\|_{L^2(0,T; L^2(\Omega))}^2 \right),
\end{align*}
for some $C=C\left(\|\nabla u_{1,0}\|, \|\nabla u_{2,0}\|,\|f_1\|_{L^2(0,T;L^2(\Omega))}, \|f_2\|_{L^2(0,T;L^2(\Omega))} \right)>0.$
\end{Theorem}

\begin{proof}
Let us add and subtract $u_1$ and $u_2$ to get 
\begin{align*}
\|u_{N_1} - u_{N_2}\| &= \|u_{N_1} - u_1+ u_2 - u_{N_2} + u_1 - u_2\| \\&\le \|u_{N_1} - u_1\| + \|u_2 - u_{N_2}\| + \|u_1 - u_2\|.
\end{align*}
Using Theorem \ref{thm:err-aux-1} and Lemma \ref{lem:aux-4-stabScheme}, the required result follows. 
\end{proof}

\section{Numerical Experiments}   \label{sec:NE}
For the numerical experiments, we consider the two-dimensional Burgers’ equation in both the non-stationary (Section~\ref{sec:nonstat-exp}) and stationary (Section~\ref{sec:stat-exp}) cases, with the details of the PDE provided in their respective sections.

For the PINN architecture, we employ a fully connected feed-forward neural network (FNN) to approximate the solution for both cases. The FNN consists of two hidden layers, each containing 32 neurons and uses \textrm{tanh} as activation function. The training data includes 8,000 points sampled from the interior, 500 points from the boundary, and 500 points from the initial time $t=0$. Training is performed in two stages: we first use the Adam optimizer for 3,000 epochs, followed by the L-BFGS optimizer for an additional 5,000 epochs. To further enhance accuracy, we apply the residual-based adaptive refinement (RAR) method, a widely used technique for solving the Burgers equation with PINNs \cite{lulu-rar}. Specifically, we randomly generate 100,000 points within the domain to compute the PDE residual and iteratively add points until the mean residual falls below $5\times 10^{-3}$.  As before, we train the network using Adam first and then refine it further with L-BFGS. We implement the PINN and train the network using the DeepXDE library~\cite{lu2021deepxde}. For testing, we sample 1,000 equally spaced points in the spatial domain to evaluate the error bounds.

\subsection{Stationary case}\label{sec:stat-exp}
For the stationary case, we consider the Burgers equation with Dirichlet boundary conditions on the domain $\Omega = [0,1]\times [0,1]$ with the forcing term 
\begin{align*}
	f(x_1,x_2) = - 2 \pi \sin(2 \pi x_1) \sin(2 \pi x_2) &\left[  \cos(2\pi x_1) \sin(2 \pi x_2) + \cos(2\pi x_2 )\sin(2\pi x_1) + 4 \pi \nu  \right].
\end{align*}

For this setup, equation \eqref{eqn:St-Burg} admits the exact solution
\begin{equation}\label{eq:per_true_sol_stat}
	u(x_1,x_2) =  \sin(2 \pi x_1)\, \sin(2 \pi x_2),
\end{equation}
serving as the reference  In our experiments, we use $\nu = \pi/4$.

Figure \ref{fig:truevspredict-stat} compares the true solution from equation \eqref{eq:per_true_sol_stat} with the PINN approximation. The left panel illustrates the exact solution, the middle panel shows the corresponding PINN-generated approximation and right panel illustrated the absolute error between the two. The $H^1$-error and residual error in this case are $5.4862 \times 10^{-4}$ and $3.2998 \times 10^{-3}$, resp.

\begin{figure}[h!]
	\includegraphics[width=0.2\linewidth]{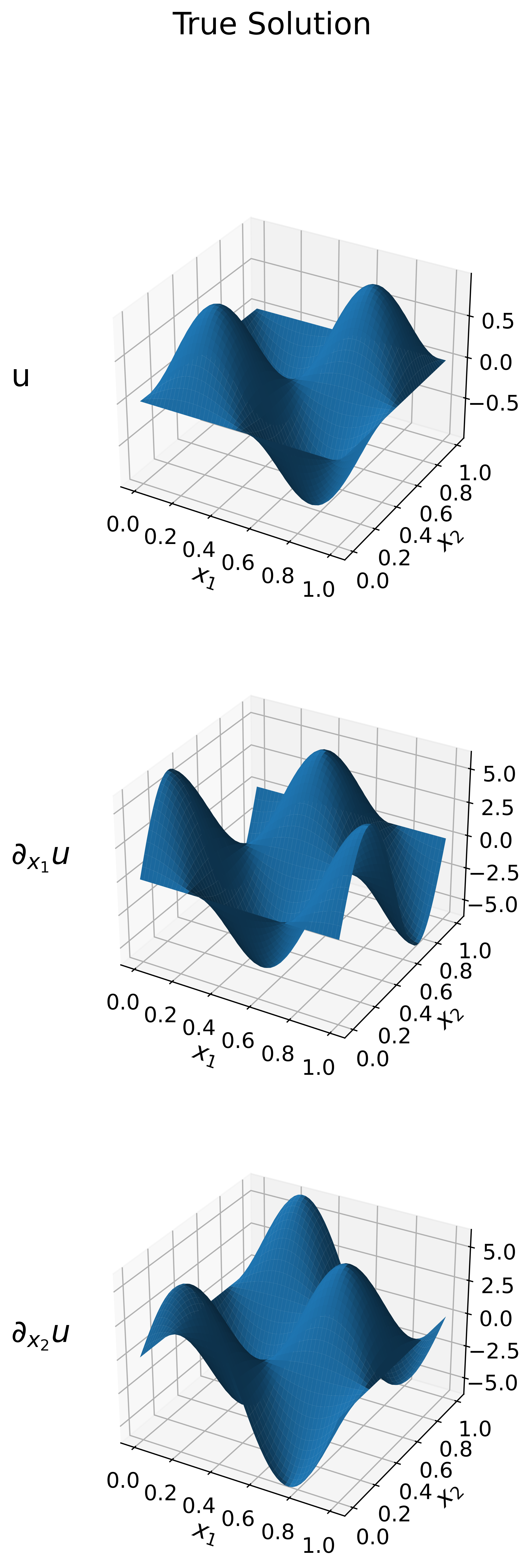}\,
	\includegraphics[width=0.2\linewidth]{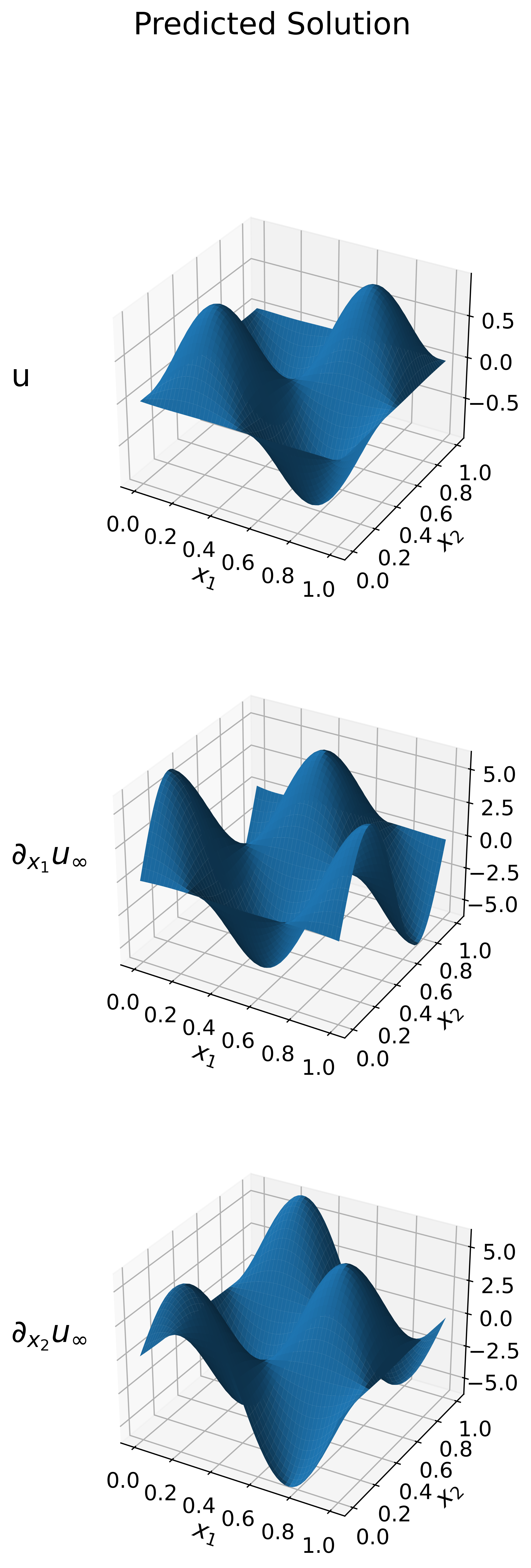}\,
	\includegraphics[width=0.43\linewidth]{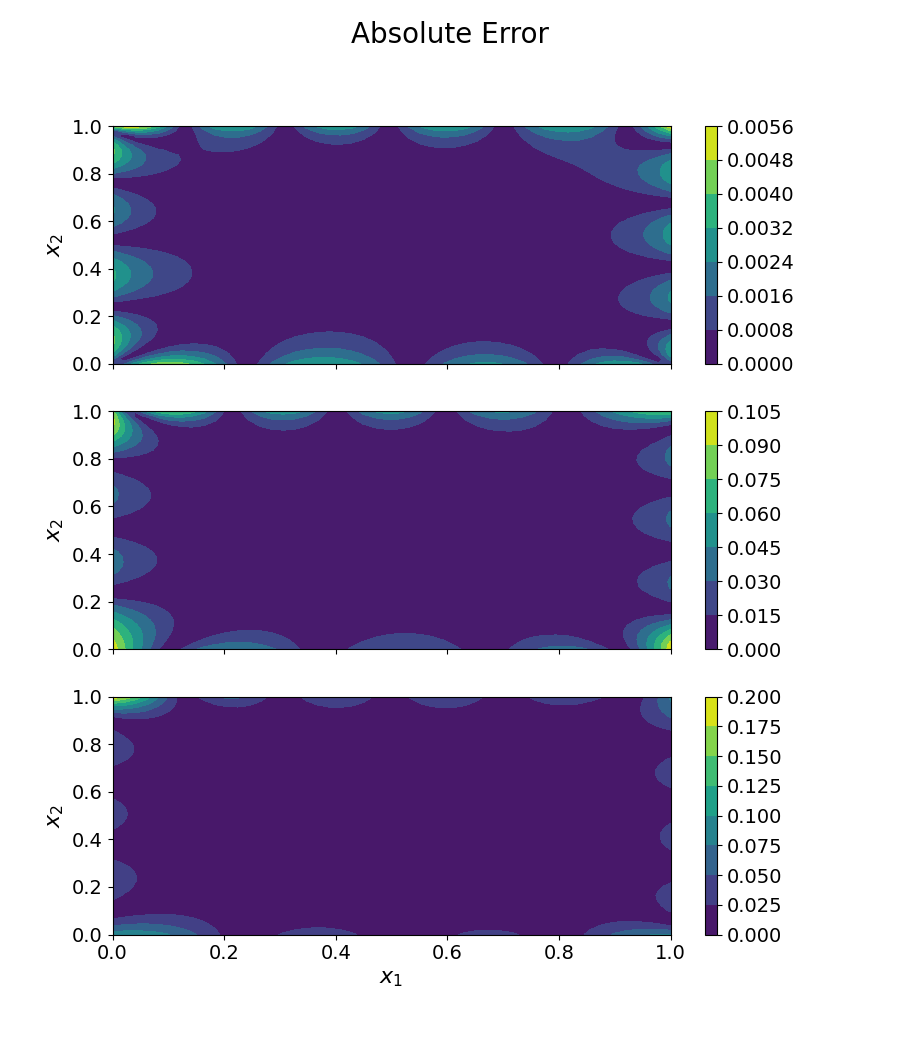}
	\caption{(Left) Comparison between the true solution from equation \eqref{eq:per_true_sol_stat} (left) and the PINN approximation (middle) along with the absolute error (right).}
	\label{fig:truevspredict-stat}
\end{figure}

\subsection{Non-stationary case}\label{sec:nonstat-exp}
We consider the Burgers equation with periodic boundary conditions on the domain $\Omega = [0,1]\times [0,1]$ with the forcing term 
\begin{align*}
	f(x_1,x_2,t) = - 2 \pi \exp(-4\pi^2 \nu t) &\left[( \cos(2\pi (x_1-t)) + \cos(2\pi (x_2-t) )  ) \right. \\ &\left. (1 - \exp(-4\pi^2 \nu t)(\sin(2 \pi (x_1-t)) + \sin(2 \pi (x_2-t))) ) \right]
\end{align*}
and initial condition $u(x_1,x_2,0) = \sin(2 \pi (x_1))\, \sin(2 \pi (x_2))$.

For this setup, equation \eqref{eq:Burg} admits the exact solution
\begin{equation}\label{eq:per_true_sol}
	u(x_1,x_2,t) = \exp(-4\pi^2 \nu t)\, \sin(2 \pi (x_1-t))\, \sin(2 \pi (x_2-t)),
\end{equation}
which serves as a reference to validate our error estimates. In our experiments, we use $\nu = 0.01$.

Figure \ref{fig:truevspredict} compares the true solution from equation \eqref{eq:per_true_sol} with the PINN approximation at three time stamps $t = 0,\,0.5$ and $1$. The left panel illustrates the exact solution, while the right panel shows the corresponding PINN-generated approximation. 

\begin{figure}[h!]
	\includegraphics[width=0.45\linewidth]{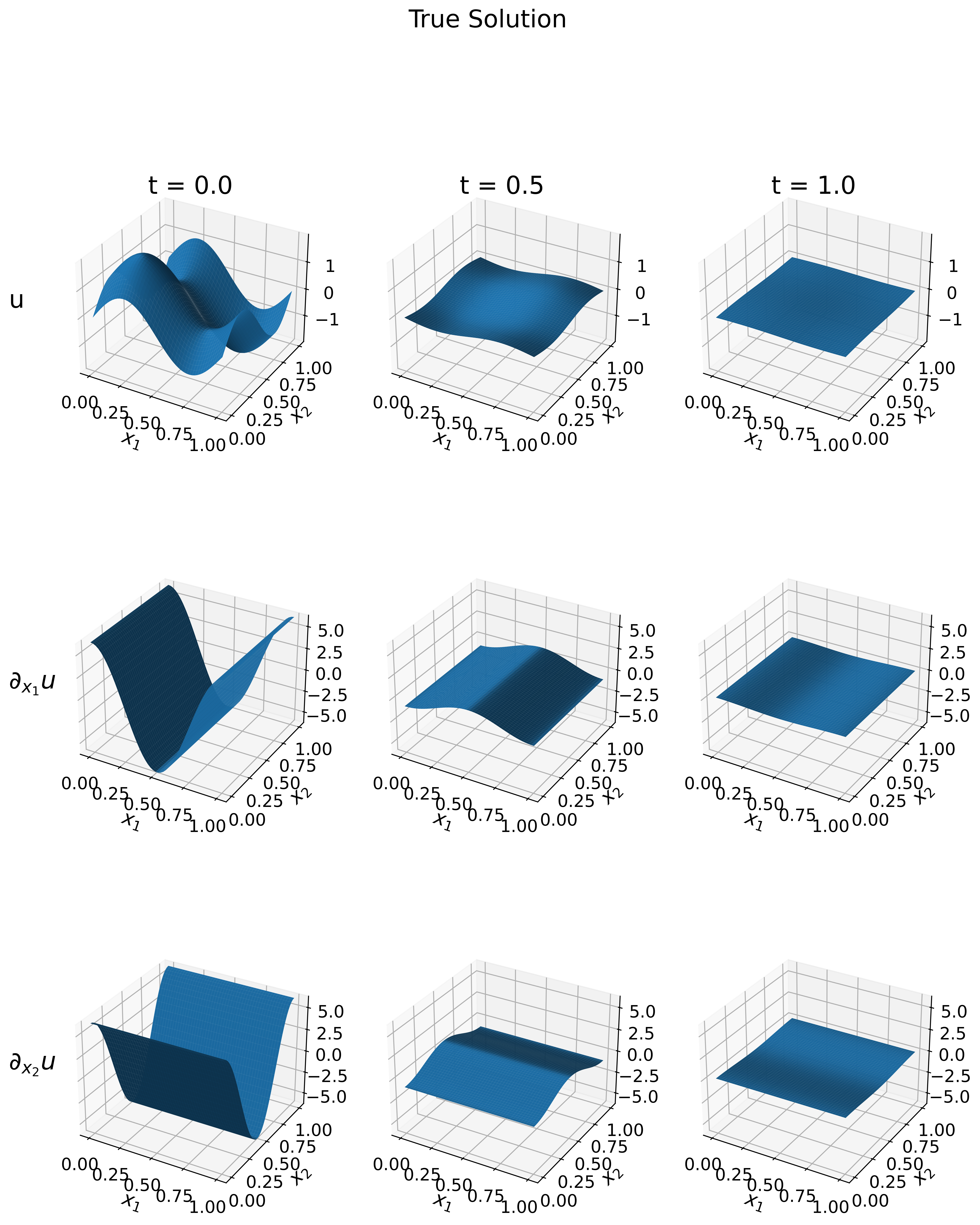}\,
	\includegraphics[width=0.45\linewidth, trim={2 0 0 0},clip=True]{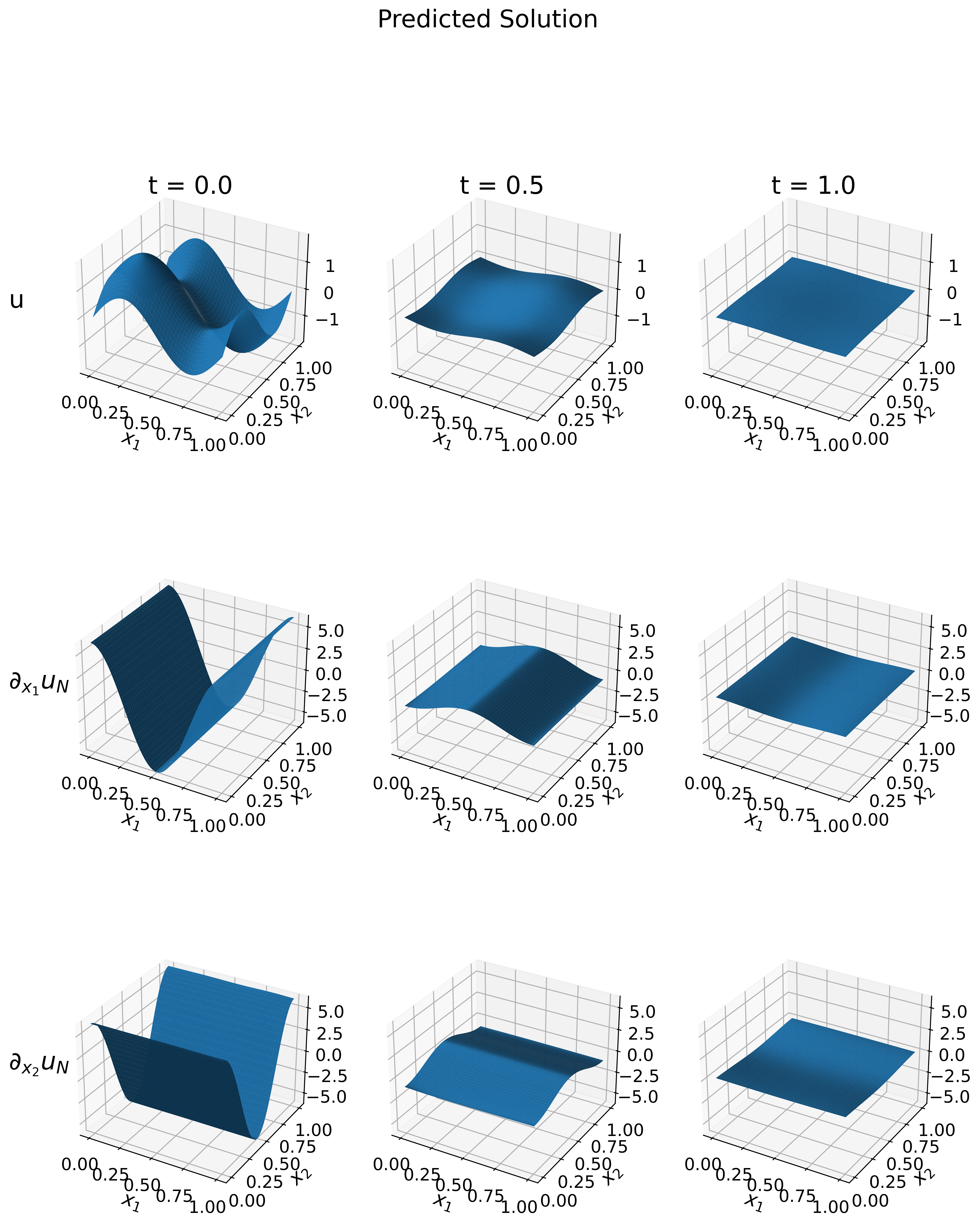}
		\caption{(Left) Comparison between the true solution from equation \eqref{eq:per_true_sol} (left) and the PINN approximation (right) at $t =  0, 0.5$ and $1$.}
		\label{fig:truevspredict}
\end{figure}

Figure \ref{fig:abserror} presents the absolute error between the true solution and the PINN approximation at $t = 0,\,0.5$ and $1$. The error distribution highlights regions where the approximation deviates from the exact solution, providing insight into the model’s accuracy.

\begin{figure}[h!]
	\includegraphics[width=0.45\linewidth]{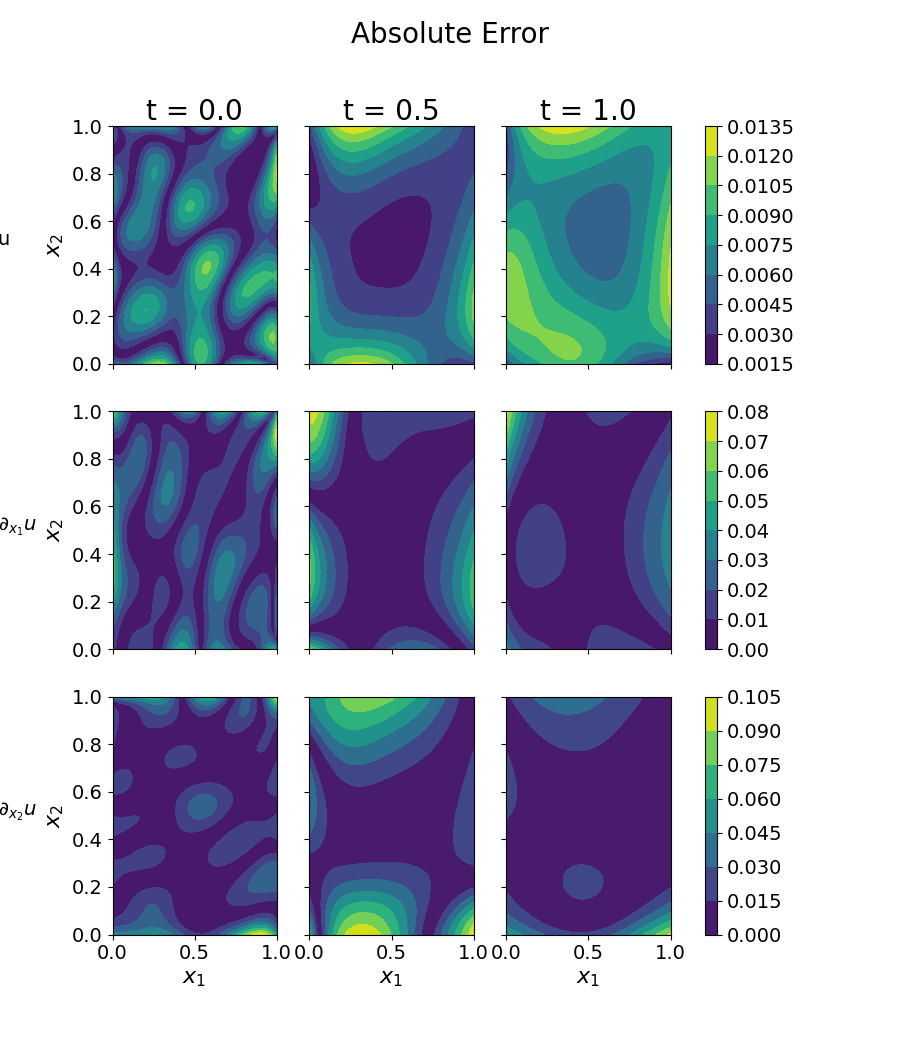}
	\caption{Absolute error between the true solution from equation \eqref{eq:per_true_sol} and the PINN approximation at $t =  0, 0.5$ and $1$.}
	\label{fig:abserror}
	\end{figure}
	
Figure \ref{fig:H1vsres} depicts the $H^1$ error (dashed line) and the residual loss (solid line) as functions of time. While the exact value of the constant $C$ in Theorem~\ref{thm:err-aux-1} is unknown, the figure suggests that the two quantities are not significantly different, aligning with the behavior described in the theorem.
\begin{figure}[h!]
	\centering
	\begin{minipage}{0.45\linewidth}
		\centering
		\begin{filecontents*}{figures/nonstat/H1_error_residual.csv}
Time,H1 Error,Residual
0.0,0.00012997037334525633,0.0024669997380325106
0.10101010101010102,0.00017355940105258532,0.002729058337790106
0.20202020202020204,0.0002899378753153246,0.002272169230758961
0.30303030303030304,0.000458334631796974,0.00194947954642084
0.4040404040404041,0.0005934961147724552,0.0018041920976360735
0.4141414141414142,0.000599283299983189,0.0017918745349601117
0.5050505050505051,0.00059632098836157,0.0015857573925086538
0.6060606060606061,0.0006009283882518039,0.0016022427550234056
0.7070707070707072,0.0006641040845670833,0.0015850753229902758
0.8080808080808082,0.000631662268892456,0.0016936410156359245
0.9090909090909092,0.0004893333774101593,0.0017094236183491064
1.0,0.0005083137709281918,0.0022268232443724023
\end{filecontents*}
\pgfplotstabletypeset[
		col sep=comma,
		header=true,
		every head row/.style={before row=\hline, after row=\hline},
		every last row/.style={after row=\hline}
		]{figures/nonstat/H1_error_residual.csv}
	\end{minipage}%
	\hspace{1cm} 
	\begin{minipage}{0.45\linewidth}
		\centering
		\includegraphics[width=\linewidth]{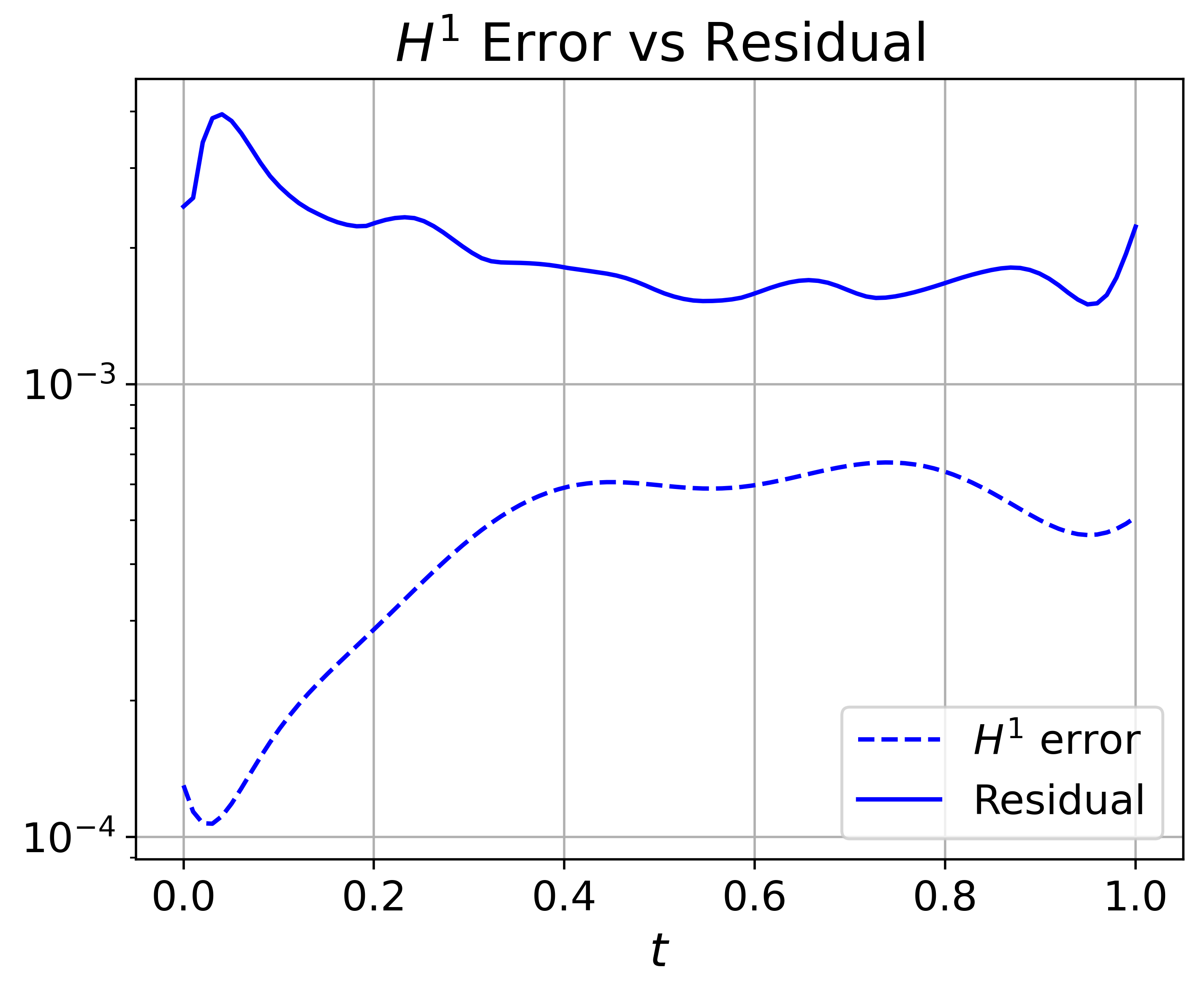}
	\end{minipage}
	\caption{$H^1$ error between the true solution from equation \eqref{eq:per_true_sol} and the PINN approximation, along with the residual loss as a function of time, presented in a table (left) and as a figure (right).}
	\label{fig:H1vsres}
\end{figure}


\bibliographystyle{amsplain}
\bibliography{References}

\end{document}